\documentclass[onefignum,onetabnum]{siamart171218}
\usepackage[left=2.9cm,right=2.9cm]{geometry}
\usepackage[utf8]{inputenc}
\DeclareUnicodeCharacter{FF0E}{.}

\usepackage{xcolor}

\newcommand\norm[1]{\left\lVert#1\right\rVert}
\newcommand\abs[1]{\left|#1\right|}
\newcommand{\op}{\overline{\partial}}
\newcommand{\Span}{\mbox{Span}\,}
\newcommand*\diff{\mathop{}\!\mathrm{d}}
\newcommand{\R}{\mathbb{R}}

\usepackage{lipsum}
\usepackage{amsfonts}
\usepackage{graphicx}
\usepackage{epstopdf}
\usepackage{algorithmic}
\usepackage{cleveref}

\ifpdf
  \DeclareGraphicsExtensions{.eps,.pdf,.png,.jpg}
\else
  \DeclareGraphicsExtensions{.eps}
\fi

\newsiamremark{remark}{Remark}
\newsiamremark{hypothesis}{Hypothesis}
\crefname{hypothesis}{Hypothesis}{Hypotheses}
\newsiamthm{claim}{Claim}

\headers{PARAMETRIC POD-ROM SENSITIVITY FOR G-EQUATION MODELS}{S. MA, L. XUE, AND Z. ZHANG}

\title{Parametric Sensitivity of POD Reduced-Order Models for Semilinear Evolution Equations with Applications to G-Equations\thanks{Submitted to the editors DATE.
\funding{S. Ma was supported by the HKU Presidential PhD Scholar Programme. L. Xue was supported by the Hong Kong PhD Fellowship Scheme. Z. Zhang was partially supported by the National Natural Science Foundation of China (Projects 92470103 and 12171406), the Hong Kong RGC grant (projects 17304324 and 17300325), the Seed Funding Programme for Basic Research (HKU), and the Hong Kong RGC Research Fellow Scheme 2025.}}}

\author{
  Shengbo Ma\thanks{Department of Mathematics, The University of Hong Kong, Hong Kong, China.  \href{mailto:u3011677@connect.hku.hk}{\texttt{u3011677@connect.hku.hk}}}
  \and
  Luhao Xue\thanks{Corresponding author. Department of Mathematics, The University of Hong Kong, Hong Kong, China.  \href{mailto:luhxue214@connect.hku.hk}{\texttt{luhxue214@connect.hku.hk}}}
  \and
  Zhiwen Zhang\thanks{Corresponding author. Department of Mathematics, The University of Hong Kong, Hong Kong, China. Materials Innovation Institute for Life Sciences and Energy (MILES), HKU-SIRI, Shenzhen, 518045, P.R. China. zhangzw@hku.hk. \href{mailto:zhangzw@hku.hk}{\texttt{zhangzw@hku.hk}}}
}

\usepackage{amsopn}

\makeatletter
\newcommand*{\addFileDependency}[1]{
  \typeout{(#1)}
  \@addtofilelist{#1}
  \IfFileExists{#1}{}{\typeout{No file #1.}}
}
\makeatother

\begin{document}

\maketitle

\begin{abstract}
Proper Orthogonal Decomposition (POD) provides low-dimensional surrogate models of evolution equations from solution snapshots. In parameterized problems, however, a basis computed at one parameter value need not remain accurate at another, while recomputing the basis for every parameter in a many-query study is costly. We develop a sensitivity analysis for POD reduced-order models of a class of parameterized semilinear evolution equations that includes the viscous G-equation and a viscous strain G-equation in their mean-free formulations. The analysis is based on a cross-space Lipschitz condition for the nonlinear operator, which accommodates the nonsmooth, gradient-dependent nonlinearities of these flame-propagation models, and on continuity moduli quantifying the parameter dependence of the bilinear form and of the nonlinearity. We prove that when a POD basis constructed at a reference parameter is applied to nearby query parameters, the resulting error variation is controlled by the corresponding parameter modulus, with constants independent of the POD dimension, the number of time steps, and the perturbation magnitude. Numerical experiments are consistent with the predicted modulus-dependent sensitivity behavior, and illustrate the practical robustness of basis reuse for nearby parameters.

\end{abstract}

\begin{keywords}
  Proper orthogonal decomposition method, viscous G-equation, strain G-equation, Hamilton-Jacobi type equation, convergence analysis.
\end{keywords}

\begin{AMS}
  35K58, 65M15, 65M60, 70H20, 80A25.
\end{AMS}

\section{Introduction}

Proper Orthogonal Decomposition (POD) is a widely used model reduction technique for partial differential equations, due to its ability to extract dominant structures from high-fidelity solution data and construct accurate low-dimensional surrogates at substantially reduced computational cost \cite{Atwell2004,BerkoozHolmesLumley1993,BORGGAARD2011269,gunzburger2017ensemble,kunisch1999control,kunisch2001galerkin,ravindran2000reduced,Sirovich1987,volkwein2013proper,wang2019proper,zhang2025novel,ma2023quasi,wang2023data}. In this work, we study the parametric sensitivity of POD reduced-order models for a class of parameterized semilinear evolution equations. More precisely, we consider systems that can be written in the abstract form \eqref{eq:intro_abstract}, which encompasses problems ranging from the linear convection--diffusion equation to nonlinear flame-front models of Hamilton--Jacobi type:
\begin{equation}
\label{eq:intro_abstract}
\langle \partial_t u(t), \psi \rangle_{V^*,V} + a(u(t),\psi;\mu) + \langle F(u(t);\mu),\psi\rangle_H
= \langle f(t),\psi\rangle_H, \qquad \forall \psi \in V, \ \text{for a.e. } t\in(0,T).
\end{equation}
Here \(V\hookrightarrow H\cong H^*\hookrightarrow V^*\) is a
Gelfand triple, with \(V\) densely and continuously embedded in \(H\);
\(\mu\in \mathcal P\subset(0,\infty)\) is the parameter of interest, \(a(\cdot,\cdot;\mu):V\times V\to\mathbb R\) is a parameter-dependent bilinear
form, \(F( \cdot ;\mu):V\to H\) is a parameter-dependent nonlinear operator,
and \(f\in C([0,T];H)\) is a prescribed forcing term.

For a fixed parameter value, POD reduced-order models have been extensively studied and have proved effective in reducing the computational complexity of large-scale dynamical systems \cite{Atwell2004,kunisch2001galerkin,kunisch2002galerkin,kunisch2004hjb,volkwein2013proper,WU2024115659}. The parametric setting, however, poses additional challenges. In many applications, such as uncertainty quantification \cite{babuska2007stochastic,MA2020112635,nobile2008sparse,xiu2005high} and digital twins \cite{crespi2023digital,nasem2024foundational,xiu2025computational}, the same model must be solved repeatedly for a range of different parameter values, and repeated high-fidelity simulations are often costly \cite{Amsallem2008,BennerGugercinWillcox2015}. Reliable reduced-order approximations are therefore essential. Since a standard POD basis is constructed from solution snapshots collected at a single reference parameter, the resulting reduced space depends on that reference value through the snapshot ensemble. When the parameter varies, the accuracy of the reduced space may therefore deteriorate, while recomputing a new POD basis for every parameter value is itself computationally expensive. Hence, a fundamental question arises regarding the generalizability of a POD basis for a specific reference parameter: does it still yield accurate approximations under parameter perturbations, and can the resulting degradation be bounded quantitatively in terms of the parameter deviation?

Answering this question has implications beyond the reuse of a reference basis. Indeed, rigorous a priori sensitivity information of this form may serve as an indicator for greedy enrichment, adaptive sampling, and parameter-local basis libraries \cite{Akkari2014,Eftang2010,Grepl_Patera_2005,Haasdonk2011}.
While such parametric sensitivity has been studied for equations with relatively standard nonlinear structures, such as the Navier--Stokes and Burgers equations \cite{AKKARI2014951,AKKARI2014522,Akkari2014}, the corresponding theory for semilinear evolution equations with gradient-dependent nonlinearities remains largely underdeveloped. Consequently, the mathematical justification for basis reuse is still incomplete for this important class of systems.

A primary physical motivation for addressing this theoretical gap comes from the viscous G-equation in turbulent combustion \cite{Osher2004, peters2000turbulent, Jackxin2009}. This equation involves a nonlinearity that depends explicitly on the gradient of the solution and thus fits naturally into the class of problems considered here. Recent POD-based computations for the viscous G-equation have shown that reduced-order modeling can yield substantial computational savings for fixed parameter values \cite{gu2021error}. While this foundational work provided a priori error estimates, the resulting bounds involve constants that depend on the dimension of the POD space. The analysis developed in our paper not only extends these results to the parametric sensitivity setting but also removes this dimension dependence.

Beyond this primary motivation, the abstract mathematical formulation in this work is not restricted to the viscous G-equation alone. It also covers simpler benchmark problems such as the convection--diffusion equation, and, as a further illustration, a viscous strain G-equation \cite{LIU201320,Xin2014FrontQI} obtained by linearizing the curvature term in the original strain model. To our knowledge, rigorous reduced-order analysis for the original strain G-equation remains limited, largely due to the complexity of its curvature term. Nevertheless, the viscous strain model considered here retains the nonlinear, strain-induced nonconvex gradient dependence of the original formulation, showing that the present framework can also accommodate more structured and complex gradient-dependent effects.

The main objective of this paper is to develop a sensitivity analysis for POD reduced-order models associated with the abstract evolution equation \eqref{eq:intro_abstract}, thereby providing a priori justification for basis reuse. A key point of our analysis is the use of a cross-space Lipschitz condition for the nonlinear operator \(F\), which is tailored to gradient-dependent nonlinearities. In addition, the dependence on the parameter is described through a continuity modulus $\rho_{\mathcal P}$ for both the bilinear form and the nonlinear operator. Under these assumptions, we prove that a POD basis constructed at a reference parameter \(\mu_0\) can be reused at a query parameter \(\mu_1\), with the resulting error controlled by the reference-parameter POD error and a parameter-sensitivity term governed by \(\rho_{\mathcal P}(|\mu_1-\mu_0|)\); see Theorem~\ref{thm:main theorem}. The constants in this estimate are independent of the POD dimension $\ell$, the number of time steps $m$, and the perturbation magnitude $|\mu_1- \mu_0|$. The estimate shows that a reference POD basis can be reused locally when the reference-parameter POD error is small and the parameter modulus is small. In this sense, the result provides a reliability criterion for local basis reuse in many-query parameter studies.

To assess how well the estimate reflects the observed parameter sensitivity, we present numerical experiments for several representative problems within our framework, including the benchmark convection--diffusion equation under Dirichlet boundary conditions, as well as the viscous and viscous strain G-equations under periodic boundary conditions and in the mean-free formulation. The diffusion and Markstein-number perturbation tests
exhibit the expected Lipschitz-type behavior. We also include a flow-amplitude test in the viscous G-equation, where a non-viscosity parameter enters through a nonlinear parameterization. These numerical tests are consistent with the modulus-dependent sensitivity behavior predicted by the theory and support the robustness of basis reuse in nearby parameter regimes.

The rest of the paper is organized as follows. \Cref{Problem Settings} presents the problem settings, including the assumptions on the bilinear and nonlinear terms and the construction of the POD reduced-order space. The main theoretical results are given in \Cref{sec:error_and_sensitivity}, where
parameter sensitivity estimates are combined with POD error estimates to derive the modulus-controlled error bound for query parameters. \Cref{section: numerical part} presents numerical experiments illustrating the parameter-sensitivity behavior predicted by the theoretical estimates. Finally, \Cref{sec:conclusions} provides conclusions and discusses potential future extensions. Technical proofs, including the coercivity of the bilinear forms, the cross-space Lipschitz and parameter-continuity estimates for the nonlinear operators, are deferred to \Cref{section: Verification}.

\section{Problem Settings}
In this section, we establish the mathematical framework for our sensitivity analysis of POD.  
\label{Problem Settings}

\subsection{Problem Formulation}
\label{subsec:problem_formulation}

Let \(V\) and \(H\) be real, separable Hilbert spaces with \(V\) densely and compactly embedded in \(H\). Identifying \(H\) with its dual \(H^*\), we obtain the Gelfand triple 
\[
    V \hookrightarrow H \cong H^* \hookrightarrow V^*,
\]
where all embeddings are continuous and dense. We denote the inner products on \(H\) and \(V\) by
\(\langle\cdot,\cdot\rangle_H\) and
\(\langle\cdot,\cdot\rangle_V\), respectively, with induced norms
\(\norm{\cdot}_H\) and \(\norm{\cdot}_V\). The duality pairing between
\(V^*\) and \(V\) is denoted by
\(\langle\cdot,\cdot\rangle_{V^*,V}\). The continuous embedding
\(V\hookrightarrow H\) implies that there exists a constant \(\alpha>0\) such
that
\begin{equation}
    \norm{\varphi}_H^2
    \le
    \alpha \norm{\varphi}_V^2,
    \qquad \forall \varphi\in V .
\end{equation}

For a fixed time horizon \(T>0\), we define
\[
    L^2(0,T;V)
    :=
    \left\{
        u:[0,T]\to V\ \text{measurable}:
        \int_0^T \norm{u(t)}_V^2 \diff t <\infty
    \right\},
\]
equipped with the norm
\[
    \norm{u}_{L^2(0,T;V)}
    :=
    \left(
        \int_0^T \norm{u(t)}_V^2 \diff t
    \right)^{1/2}.
\]
We also introduce the evolution space
\[
    W(0,T;V)
    :=
    \left\{
        u\in L^2(0,T;V):
        \partial_t u\in L^2(0,T;V^*)
    \right\},
\]
where the time derivative is understood in the distributional sense. For
\(u\in W(0,T;V)\), the standard energy identity reads
\begin{equation}
    \langle \partial_t u(t),u(t)\rangle_{V^*,V}
    =
    \frac12\frac{\diff}{\diff t}\norm{u(t)}_H^2
    \quad
    \text{for a.e. } t\in(0,T).
\end{equation}

We fix a compact admissible parameter set $\mathcal P \subset (0,\infty)$. For notational simplicity, the parameter is taken to be scalar; the same argument extends to compact subsets of $\mathbb R^{d}$, up to notational changes. For each \(\mu\in\mathcal P\), we consider the following semilinear evolution problem: find \(u_\mu\in W(0,T;V)\) such that
\begin{subequations}\label{PDE:nonlinear}
    \begin{align}
        \langle \partial_t u_\mu(t),\psi\rangle_{V^*,V}
        +
        a(u_\mu(t),\psi;\mu)
        +
        \langle F(u_\mu(t);\mu),\psi\rangle_H
        &=
        \langle f(t),\psi\rangle_H,
        \quad
        \forall \psi\in V,\  \text{for a.e. }t\in(0,T),
        \\
        u_\mu(0)&=\phi .
    \end{align}
\end{subequations}
Here \(f\in C([0,T];H)\) is a prescribed forcing term, \(\phi \in V\) is the initial datum, \(a:V\times V\times \mathbb R_{>0}\to\mathbb R\) is a parameter-dependent bilinear form, and \(F:V\times \mathbb R_{>0}\to H\) is a nonlinear operator.

For the bilinear form \(a\), we assume that there exist continuous functions
    $\beta:\mathcal P\to (0,\infty),
    \kappa:\mathcal P\to (0,\infty),$
and a modulus of continuity $\omega_a^{\mathcal P}:[0,\infty)\to[0,\infty)$,
such that, for all \(\mu,\nu\in\mathcal P\) and all
\(\varphi,\psi\in V\), the following conditions hold.
\begin{itemize}
    \item[(A1)] \textbf{Continuity:}
    \begin{equation}
        \abs{a(\varphi,\psi;\mu)}
        \le
        \beta(\mu)\norm{\varphi}_V\norm{\psi}_V .
    \end{equation}

    \item[(A2)] \textbf{Coercivity:}
    \begin{equation}
        a(\varphi,\varphi;\mu)
        \ge
        \kappa(\mu)\norm{\varphi}_V^2 .
    \end{equation}

    \item[(A3)] \textbf{Parameter continuity:}
    \begin{equation}
        \abs{a(\varphi,\psi;\mu)-a(\varphi,\psi;\nu)}
        \le
        \omega_a^{\mathcal P}(\abs{\mu-\nu})
        \norm{\varphi}_V\norm{\psi}_V .
    \end{equation}
\end{itemize}
Here a modulus of continuity means a nondecreasing function
\(\omega:[0,\infty)\to[0,\infty)\) satisfying
\[
    \omega(0)=0,
    \qquad
    \lim_{r\to 0}\omega(r)=0.
\]

For the nonlinear operator \(F:V\times \mathbb R_{>0}\to H\),  we assume that the following conditions hold.

\begin{itemize}
    \item[(F1)] \textbf{Cross-space Lipschitz continuity in the state variable:}
    There exists a continuous function
    \(L_F:\mathcal P\to[0,\infty)\) such that, for every
    \(\mu\in\mathcal P\) and all \(u,v\in V\),
    \begin{equation}\label{lipschitz in function}
        \norm{F(u;\mu)-F(v;\mu)}_H
        \le
        L_F(\mu)\norm{u-v}_V .
    \end{equation}

    \item[(F2)] \textbf{Parameter continuity:}
    There exists a modulus of continuity
    \(\omega_F^{\mathcal P}:[0,\infty)\to[0,\infty)\) such that,
    for every \(\mu,\nu\in\mathcal P\) and every \(u\in V\),
    \begin{equation}\label{modulus in parameter}
        \norm{F(u;\mu)-F(u;\nu)}_H
        \le
        \omega_F^{\mathcal P}(\abs{\mu-\nu})
        \left(\norm{u}_V+1\right).
    \end{equation}
\end{itemize}

The constant \(+1\) in \eqref{modulus in parameter} allows the nonlinear operator to depend on the parameter even at the zero state. In particular, the framework can accommodate parameter-dependent lower-order terms or source-type contributions for which \(F(0;\mu)\) varies with \(\mu\).

Combining (F1) and (F2), we obtain, for all \(\mu,\nu\in\mathcal P\) and
all \(u,v\in V\),
\begin{equation}\label{lipschitz combined}
    \norm{F(u;\mu)-F(v;\nu)}_H
    \le
    \omega_F^{\mathcal P}(\abs{\mu-\nu})(\norm{u}_V+1)
    +
    L_F(\nu)\norm{u-v}_V .
\end{equation}

For every \(\mu\in\mathcal P\), we assume that problem
\eqref{PDE:nonlinear} admits a unique solution
\(u_\mu\in W(0,T;V)\cap C([0,T];H)\) with
\(\partial_t u_\mu\in L^2(0,T;V^*)\). For the subsequent sensitivity analysis, we further assume higher regularity, specifically that $u_\mu \in W^{2,2}(0,T;H) \cap C([0,T];V)$. Under this regularity assumption, \(\partial_t u_\mu\in C([0,T];H)\), and hence the duality pairing in \eqref{PDE:nonlinear} is identified with the \(H\)-inner product in the estimates below. Such regularity is natural for viscous equations and is standard in POD-based analyses; see, e.g., $\cite{gu2021error,kunisch2001galerkin,kunisch2002galerkin}$.

\begin{remark}
Condition (F1) employs a cross-space  Lipschitz formulation that is mathematically weaker than the standard Lipschitz requirements typically found in the literature. Specifically, condition (F1) bounds the $H$-norm of the output difference $\|F(u;\mu) - F(v;\mu)\|_H$ by the $V$-norm of the input difference $\|u - v\|_V$. This relaxed requirement is specifically designed for our target applications: it naturally accommodates the nonsmooth and gradient-dependent nonlinearities inherent in the viscous and strain G-equations. This avoids the overly restrictive global bounds that would be required if both norms were taken in the same functional space (e.g., from $H^1$ to $H^1$ or $L^2$ to $L^2$).
\end{remark}

\begin{remark}
Furthermore, regarding the coercivity assumption (A2), while it is naturally satisfied under homogeneous Dirichlet boundary conditions, we emphasize that for systems equipped with periodic boundary conditions, such as the G-equations considered in our applications, the standard bilinear form associated with the diffusion operator lacks coercivity over the entire $H^1(\Omega)$ space due to the null space of constant functions. By employing the mean-free treatment, the functional space is properly restricted to the subspace $\mathring{V} = \{ \psi \in H_{per}^1(\Omega) : \int_\Omega \psi \diff x = 0 \}$. Over this specific subspace, the Poincar\'e inequality holds, which provides the strict mathematical justification for the coercivity condition (A2) required by our framework. The mean component can then be recovered from the mean-free component through the corresponding spatial averaging relation. Thus, restricting the analysis to the mean-free subspace isolates the coercive part of the problem without creating an additional closure difficulty for the reduced-order model.

\end{remark}

\begin{remark}
As a final note on the problem formulation, we emphasize that our assumptions accommodate the highly nonlinear physical models motivating this study. The mathematical verification demonstrating that both the viscous G-equation and the viscous strain G-equation strictly satisfy the abstract conditions (A1)--(A3) and (F1)--(F2) can be found in \Cref{section: Verification}.
\end{remark}

\subsection{Construction of POD Basis}\label{subsection:general PODconstruction}

We employ the Proper Orthogonal Decomposition (POD) method to construct a reduced-order model for the nonlinear PDE \eqref{PDE:nonlinear}. The key idea is to extract the dominant modes from solution snapshots and project the dynamics onto a low-dimensional subspace that captures the system's essential features.

Fix a reference parameter \(\mu_0\in\mathcal P\), and let
\(u_{\mu_0}\) denote the corresponding solution of \eqref{PDE:nonlinear}.
The POD basis is constructed from snapshots of \(u_{\mu_0}\). Consider the set of numerical solutions $\{u_{\mu_0}(t_k)\}_{k=0}^m$ evaluated at discrete time instances $t_k = k\Delta t$, where $\Delta t = T/m$ denotes the time step size (which will be chosen sufficiently small in subsequent analysis). To capture both the solution behavior and its temporal evolution, we augment the snapshot set with discrete time derivatives.  This enriched snapshot ensemble ensures a better approximation of the dynamics. Specifically, we utilize \(2m+1\) snapshots defined as follows:
\[
y_j = u_{\mu_0}(t_j), \quad j = 0, \dots, m,
\]
and
\[
y_j = \op u_{\mu_0}(t_{j-m}), \quad j = m+1, \dots, 2m,
\]
where \(\op u_{\mu_0}(t_i) = (u_{\mu_0}(t_i) - u_{\mu_0}(t_{i-1}))/\Delta t\) denotes the discrete time derivative. We define the set \(S := \{y_i \mid i = 0, \dots, 2m\}\) and assume the linear span of \(S\) has dimension \(n\). Let $S^\ell = \{\psi_1(\cdot), \dots, \psi_\ell(\cdot)\}$ denote the $\ell$-dimensional orthonormal POD basis functions obtained by solving the minimization problem:
\begin{equation}\label{eqn:PODerror generalcase}
    \min_{\psi_1,\ldots,\psi_\ell\in V}
    \frac1{2m+1}\sum_{j=0}^{2m}
    \left\|
        y_j-\sum_{k=1}^\ell \langle y_j,\psi_k\rangle_V\psi_k
    \right\|_V^2
    \quad
    \text{s.t. }
    \langle \psi_i,\psi_j\rangle_V=\delta_{ij},
\end{equation}
where $\ell \in \{1, \dots, n\}$. To solve \eqref{eqn:PODerror generalcase}, we consider the eigenvalue problem
\[
Kv=\lambda v,
\]
where $K\in\R^{{2m+1}\times {2m+1}}$ and $K_{ij}=\frac{1}{2m+1}\langle y_i,y_j\rangle_{V}$ is the snapshot correlation matrix. Let \( v_{k}, k = 1, \ldots, 2m+1 \), be the eigenvectors and \( \lambda_{1} \geq \lambda_{2} \geq \cdots \geq \lambda_{n} > 0 \) be the positive eigenvalues. It has been shown in \cite{Sirovich1987,volkwein2013proper} that the solution of the optimization problem \eqref{eqn:PODerror generalcase} is explicitly given by

\begin{equation}
\psi_{k} = \frac{1}{\sqrt{(2m+1)\lambda_{k}}} \sum_{j=0}^{2m} (v_{k})_{j} y_j, \quad 1 \leq k \leq \ell.
\end{equation}

Moreover, the corresponding minimum projection error is exactly quantified by the tail of the eigenvalue spectrum
\begin{equation}
    \frac1{2m+1}\sum_{j=0}^{2m}
    \left\|
        y_j-\sum_{k=1}^\ell \langle y_j,\psi_k\rangle_V\psi_k
    \right\|_V^2
    =
    \sum_{k=\ell+1}^n\lambda_k .
\end{equation}
With the basis functions established, we define the POD reduced-order space as $V^\ell = \Span\{\psi_1, \dots, \psi_\ell\}$.

In practice, the dimension $\ell$ of the reduced-order space is determined by the decay rate of the correlation matrix eigenvalues. A standard criterion is to choose the smallest integer $\ell$ such that the relative projection error is bounded by a predefined tolerance $\varepsilon > 0$, namely,
\begin{equation}
\frac{\sum_{k=\ell+1}^n \lambda_k}{\sum_{k=1}^n \lambda_k} \leq \varepsilon.
\end{equation}

\subsection{Backward Euler-Galerkin Discretization}

We now introduce the fully discrete scheme using a Backward Euler temporal discretization combined with Galerkin projection onto the POD space. For a parameter \(\mu\in\mathcal P\), we seek
\(\{U_k^\mu\}_{k=0}^m\subset V^\ell\) of the form
\[
    U_k^\mu=\sum_{j=1}^\ell c_j^k(\mu)\psi_j,
    \qquad
    \mathbf c^k(\mu)
    :=
    \big(c_1^k(\mu),\ldots,c_\ell^k(\mu)\big)^\top
    \in\mathbb R^\ell .
\]
The Backward Euler-Galerkin scheme reads: for \(k=1,\ldots,m\), find
\(U_k^\mu\in V^\ell\) such that
\begin{subequations}\label{eqn:POD nonlinear}
    \begin{align}
    \label{eqn:POD nonlinear a}
        \left\langle
            {\op U_k^\mu},
            \psi
        \right\rangle_H
        +
        a(U_k^\mu,\psi;\mu)
        +
        \langle F(U_k^\mu;\mu),\psi\rangle_H
        &=
        \langle f(t_k),\psi\rangle_H,
        \quad \forall \psi\in V^\ell, \\
        \label{eqn:POD nonlinear b}
        \langle U_0^\mu,\psi\rangle_H
        &=
        \langle \phi,\psi\rangle_H,
        \quad \forall \psi\in V^\ell .
    \end{align}
\end{subequations}
where $\op U_k^\mu = (U_k^\mu - U_{k-1}^\mu)/\Delta t$ denotes the discrete backward difference operator for $k = 1, \dots, m$. Equation \eqref{eqn:POD nonlinear b} implies that the initial condition $U_0^\mu$ is precisely the $H$-orthogonal projection of the exact initial data $\phi$ onto the reduced-order space $V^\ell$. By substituting $U_k^\mu = \sum_{i=1}^\ell c_i^k(\mu)\psi_i$ into \eqref{eqn:POD nonlinear a} and taking the test functions as $\psi = \psi_i$ for $i=1,\dots,\ell$, the fully discrete scheme translates into the following implicit system of nonlinear algebraic equations for the coefficient vector \(\mathbf c^k(\mu)\):
\begin{equation}\label{eqn:matrix_form}
    \frac{1}{\Delta t} M \big(\mathbf c^k(\mu)-\mathbf c^{k-1}(\mu)\big) + A(\mu) \mathbf{c}^k(\mu) + \mathbf{F}(\mathbf{c}^k(\mu); \mu) = \mathbf{f}(t_k),
\end{equation}
where $M \in \R^{\ell \times \ell}$ is the mass matrix with entries $M_{ij} = \langle \psi_j, \psi_i \rangle_H$, $A(\mu) \in \R^{\ell \times \ell}$ is the parameter-dependent stiffness matrix with $A_{ij}(\mu) = a(\psi_j, \psi_i; \mu)$, $\mathbf{F} \in \R^\ell$ is the nonlinear vector with components $\mathbf{F}_i = \langle F(U_k^\mu; \mu), \psi_i \rangle_H$, and $\mathbf{f}(t_k)_i = \langle f(t_k), \psi_i \rangle_H$. At each time step, this low-dimensional nonlinear system can be efficiently solved using standard iterative algorithms, such as Newton’s method or Picard iteration.

\section{Error and Sensitivity Analysis of the Parameterized POD-ROM}
\label{sec:error_and_sensitivity}

In this section, we establish the error and sensitivity estimates for the
parameterized POD reduced-order model. The objective is to quantify the error
incurred when a POD space constructed at a reference parameter is reused to
approximate the solution at a different parameter.

Let \(\mu_0,\mu_1\in\mathcal P\), where \(\mu_0\) is the reference parameter
used to generate the POD snapshots and \(\mu_1\) is the query parameter. Set $\delta:=|\mu_1-\mu_0|$. For notational simplicity, we write $u_i(t):=u_{\mu_i}(t), i=0,1,$ where \(u_i\) denotes the exact solution of \eqref{PDE:nonlinear} at the
parameter \(\mu_i\).

Let $S^\ell=\{\psi_1,\ldots,\psi_\ell\}$ be the \(V\)-orthonormal POD basis constructed from the snapshots associated
with the reference parameter \(\mu_0\), as described in
Section \ref{subsection:general PODconstruction}. The corresponding reduced space is
$    V^\ell:=\operatorname{span}S^\ell
    =
    \operatorname{span}\{\psi_1,\ldots,\psi_\ell\}$. For \(i=0,1\), let \(\{U_k^i\}_{k=0}^m\subset V^\ell\) denote the
Backward Euler--Galerkin POD solution computed at parameter \(\mu_i\) in the
same reduced space \(V^\ell\). It satisfies, for \(k=1,\ldots,m\),
\begin{subequations}\label{eqn:POD_scheme_mu_i}
    \begin{align}
        \langle \op U_k^i,\psi\rangle_H
        +
        a(U_k^i,\psi;\mu_i)
        +
        \langle F(U_k^i;\mu_i),\psi\rangle_H
        &=
        \langle f(t_k),\psi\rangle_H,
        \qquad \forall \psi\in V^\ell,
        \\
        \langle U_0^i,\psi\rangle_H
        &=
        \langle \phi,\psi\rangle_H,
        \qquad \forall \psi\in V^\ell,
    \end{align}
\end{subequations}
where $\op U_k^i:={(U_k^i-U_{k-1}^i)}/{\Delta t}$.

Since \(\mathcal P\) is compact and \(\kappa\) and \(L_F\) are
continuous, we define
\begin{equation}
\label{kappauniform}
    \underline\kappa_{\mathcal P}
    :=
    \inf_{\mu\in\mathcal P}\kappa(\mu)>0,
    \qquad
    \overline L_{F,\mathcal P}
    :=
    \sup_{\mu\in\mathcal P}L_F(\mu)<\infty.    
\end{equation}

We also define the combined parameter modulus
$    \rho_{\mathcal P}(s)
    :=
    \omega_a^{\mathcal P}(s)
    +
    \omega_F^{\mathcal P}(s)$
    for any $s\ge0$.
Throughout this section, we assume that (A1)--(A3), (F1)--(F2), and the
regularity assumptions stated in Section \ref{subsec:problem_formulation} hold. We also assume that time step 
\(\Delta t\) is
sufficiently small so that
\begin{equation}
\label{stability_cond}
        1-\frac{\overline L_{F,\mathcal P}^{\,2}}
         {\underline\kappa_{\mathcal P}}\Delta t =  1-\frac{\overline L_{F,\mathcal P}^{\,2}}
         {\underline\kappa_{\mathcal P}}\frac{T}{m} \ge \frac12, \quad     \, \,
    1-(1
    +
    L_F(\mu_0)
    +
    \frac{L_F(\mu_0)^2}{2\kappa(\mu_0)})\Delta t\ge \frac12 .
\end{equation}

We first prove a sensitivity estimate for the exact solutions. The estimate
shows that the difference between \(u_0\) and \(u_1\) is controlled by the
parameter modulus \(\rho_{\mathcal P}(\delta)\).

\begin{lemma}\label{lem:uk-u1k}
Let \(\mu_0,\mu_1\in\mathcal P\). Then, for every
\(k\in\{0,1,\ldots,m\}\),
\begin{equation}\label{inequality:uk-u1k}
    \norm{u_0(t_k)-u_1(t_k)}_H^2
    \le
    \Lambda_{\mathcal P}\rho_{\mathcal P}(\delta)^2
    e^{E_{\mathcal P}T}
    \int_0^T
        e^{-E_{\mathcal P}s}
        \left(1+\norm{u_0(s)}_V^2\right)
    \diff s,
\end{equation}
where
$
    E_{\mathcal P}
    =
    \frac{\overline L_{F,\mathcal P}^{\,2}}
         {\underline\kappa_{\mathcal P}},
    \Lambda_{\mathcal P}
    =
    \frac{2+4\alpha}
         {\underline\kappa_{\mathcal P}}.
$

\end{lemma}

\begin{proof}
Let $e(t):=u_0(t)-u_1(t).$ Subtracting the weak formulations of \eqref{PDE:nonlinear} corresponding to
\(\mu_0\) and \(\mu_1\), we obtain, for every \(\psi\in V\) and for a.e.
\(t\in(0,T)\),
\begin{equation}\label{eqn:exact_sensitivity_difference}
\begin{aligned}
    \langle \partial_t e(t),\psi\rangle_H
    +
    a(u_0(t),\psi;\mu_0)
    -
    a(u_1(t),\psi;\mu_1)
    +
    \left\langle
        F(u_0(t);\mu_0)-F(u_1(t);\mu_1),
        \psi
    \right\rangle_H
    = 0.
\end{aligned}
\end{equation}

Taking \(\psi=e(t)\), using the energy identity, and applying the coercivity condition \eqref{kappauniform} uniformly over \(\mathcal P\), we get
\begin{equation}\label{eqn:exact_sensitivity_energy}
\begin{aligned}
    \frac12\frac{\diff}{\diff t}\norm{e(t)}_H^2
    +
    \underline\kappa_{\mathcal P}\norm{e(t)}_V^2
    \le
    &
    \abs{
        a(u_0(t),e(t);\mu_0)
        -
        a(u_0(t),e(t);\mu_1)
    }
    \\
    &
    +
    \norm{
        F(u_0(t);\mu_0)-F(u_1(t);\mu_1)
    }_H
    \norm{e(t)}_H .
\end{aligned}
\end{equation}

By the parameter-continuity assumption (A3), we have
\begin{equation*}
    \abs{
        a(u_0(t),e(t);\mu_0)
        -
        a(u_0(t),e(t);\mu_1)
    }
    \le
    \omega_a^{\mathcal P}(\delta)
    \norm{u_0(t)}_V
    \norm{e(t)}_V .
\end{equation*}
Moreover, the combined estimate \eqref{lipschitz combined} gives
\begin{equation*}
\begin{aligned}
    \norm{
        F(u_0(t);\mu_0)-F(u_1(t);\mu_1)
    }_H
    &\le
    \omega_F^{\mathcal P}(\delta)
    \left(\norm{u_0(t)}_V+1\right)
    +
    L_F(\mu_1)\norm{e(t)}_V
    \\
    &\le
    \omega_F^{\mathcal P}(\delta)
    \left(\norm{u_0(t)}_V+1\right)
    +
    \overline L_{F,\mathcal P}\norm{e(t)}_V .
\end{aligned}
\end{equation*}
Substituting these two estimates into
\eqref{eqn:exact_sensitivity_energy} yields
\begin{equation}\label{eqn:exact_sensitivity_before_young}
\begin{aligned}
    \frac12\frac{\diff}{\diff t}\norm{e(t)}_H^2
    +
    \underline\kappa_{\mathcal P}\norm{e(t)}_V^2
    \le
    &
    \omega_a^{\mathcal P}(\delta)
    \norm{u_0(t)}_V
    \norm{e(t)}_V
    \\
    &
    +
    \omega_F^{\mathcal P}(\delta)
    \left(\norm{u_0(t)}_V+1\right)
    \norm{e(t)}_H
    +
    \overline L_{F,\mathcal P}
    \norm{e(t)}_V
    \norm{e(t)}_H .
\end{aligned}
\end{equation}

We now estimate the three terms on the right-hand side of
\eqref{eqn:exact_sensitivity_before_young}. First, Young's inequality gives
\begin{equation*}
    \omega_a^{\mathcal P}(\delta)
    \norm{u_0(t)}_V
    \norm{e(t)}_V
    \le
    \frac{\underline\kappa_{\mathcal P}}{4}
    \norm{e(t)}_V^2
    +
    \frac{\omega_a^{\mathcal P}(\delta)^2}
         {\underline\kappa_{\mathcal P}}
    \norm{u_0(t)}_V^2 .
\end{equation*}
Second, using the embedding
\(\norm{e(t)}_H\le \sqrt{\alpha}\norm{e(t)}_V\), we obtain
\begin{equation*}
\begin{aligned}
    &
    \omega_F^{\mathcal P}(\delta)
    \left(\norm{u_0(t)}_V+1\right)
    \norm{e(t)}_H
    \le
    \frac{\underline\kappa_{\mathcal P}}{4}
    \norm{e(t)}_V^2
    +
    \frac{\alpha\,\omega_F^{\mathcal P}(\delta)^2}
         {\underline\kappa_{\mathcal P}}
    \left(\norm{u_0(t)}_V+1\right)^2 .
\end{aligned}
\end{equation*}
Finally,
\begin{equation*}
    \overline L_{F,\mathcal P}
    \norm{e(t)}_V
    \norm{e(t)}_H
    \le
    \frac{\underline\kappa_{\mathcal P}}{2}
    \norm{e(t)}_V^2
    +
    \frac{\overline L_{F,\mathcal P}^{\,2}}
         {2\underline\kappa_{\mathcal P}}
    \norm{e(t)}_H^2 .
\end{equation*}

Substituting these bounds into
\eqref{eqn:exact_sensitivity_before_young}, the
\(\norm{e(t)}_V^2\) terms are absorbed by the coercivity term on the left.
We obtain
\begin{equation*}
\begin{aligned}
    \frac{\diff}{\diff t}\norm{e(t)}_H^2
    \le
    &
    \frac{\overline L_{F,\mathcal P}^{\,2}}
         {\underline\kappa_{\mathcal P}}
    \norm{e(t)}_H^2
    +
    \frac{2\,\omega_a^{\mathcal P}(\delta)^2}
         {\underline\kappa_{\mathcal P}}
    \norm{u_0(t)}_V^2
    +
    \frac{2\alpha\,\omega_F^{\mathcal P}(\delta)^2}
         {\underline\kappa_{\mathcal P}}
    \left(\norm{u_0(t)}_V+1\right)^2 .
\end{aligned}
\end{equation*}
Since
   $ \left(\norm{u_0(t)}_V+1\right)^2
    \le
    2\left(1+\norm{u_0(t)}_V^2\right), \
    \omega_a^{\mathcal P}(\delta)^2
    +
    \omega_F^{\mathcal P}(\delta)^2
    \le
    \rho_{\mathcal P}(\delta)^2$,
we have
\begin{equation}\label{eqn:exact_sensitivity_differential}
    \frac{\diff}{\diff t}\norm{e(t)}_H^2
    \le
    E_{\mathcal P}\norm{e(t)}_H^2
    +
    \Lambda_{\mathcal P}\rho_{\mathcal P}(\delta)^2
    \left(1+\norm{u_0(t)}_V^2\right),
\end{equation}
where
$
    E_{\mathcal P}
    =
    \frac{\overline L_{F,\mathcal P}^{\,2}}
         {\underline\kappa_{\mathcal P}},
    \Lambda_{\mathcal P}
    =
    \frac{2+4\alpha}
         {\underline\kappa_{\mathcal P}}.
$

Because \(u_0(0)=u_1(0)=\phi\), we have \(e(0)=0\). Applying Gronwall's
inequality to \eqref{eqn:exact_sensitivity_differential} gives, for every
\(t\in[0,T]\),
\begin{equation*}
    \norm{e(t)}_H^2
    \le
    \Lambda_{\mathcal P}\rho_{\mathcal P}(\delta)^2
    e^{E_{\mathcal P}t}
    \int_0^t
        e^{-E_{\mathcal P}s}
        \left(1+\norm{u_0(s)}_V^2\right)
    \diff s.
\end{equation*}
Taking \(t=t_k\) and using \(t_k\le T\), we conclude \eqref{inequality:uk-u1k}.
\end{proof}

After establishing the sensitivity bound for the exact solution, we now analyze how perturbations in the parameter affect the numerical solutions. Specifically, we consider the sensitivity of the low-rank approximations obtained using the Backward Euler-Galerkin scheme with the same POD basis constructed from parameter $\mu_0$, but applied to both parameters $\mu_0$ and $\mu_1$.

\begin{lemma}\label{lem:Uk-U1k}
For any discrete time index \(k\in\{1,2,\ldots,m\}\), the following stability estimate holds for the difference between the POD approximations corresponding to parameters $\mu_0$ and $\mu_1$:
\begin{equation}\label{inequality:U_k-U^1_k}
    \norm{U_k^0-U_k^1}_H^2
    \le
    \Lambda_{\mathcal P}\rho_{\mathcal P}(\delta)^2
    e^{2E_{\mathcal P}T}
    \Delta t
    \sum_{q=1}^m
    e^{2E_{\mathcal P}(1-q)\Delta t}
    \left(\norm{U_q^0}_V^2+1\right),
\end{equation}
where
$
    E_{\mathcal P}
    =
    \frac{\overline L_{F,\mathcal P}^{\,2}}
         {\underline\kappa_{\mathcal P}},
    \Lambda_{\mathcal P}
    =
    \frac{2+4\alpha}
         {\underline\kappa_{\mathcal P}}.
$
\end{lemma}

\begin{proof}
Let \(\xi_k:=U_k^0-U_k^1\) for \(k=0,1,\ldots,m\). Then
\(\op \xi_k=\op U_k^0-\op U_k^1\). We begin by subtracting the discrete weak formulations \eqref{eqn:POD_scheme_mu_i} corresponding to \(i=0\) and \(i=1\), which yields for any test function \(\psi\in V^\ell\):
\begin{equation}\label{eqn:discrete_difference}
    \langle \op \xi_k,\psi\rangle_H
    +
    a(U_k^0,\psi;\mu_0)
    -
    a(U_k^1,\psi;\mu_1)
    +
    \langle F(U_k^0;\mu_0)-F(U_k^1;\mu_1),\psi\rangle_H
    =
    0.
\end{equation}
Rewriting this identity by isolating the perturbed bilinear form:
\begin{equation*}
    \langle \op \xi_k,\psi\rangle_H
    +
    a(\xi_k,\psi;\mu_1)
    =
    -
    \left(
        a(U_k^0,\psi;\mu_0)
        -
        a(U_k^0,\psi;\mu_1)
        +
        \langle F(U_k^0;\mu_0)-F(U_k^1;\mu_1),\psi\rangle_H
    \right).
\end{equation*}
Selecting the test function \(\psi=\xi_k\in V^\ell\) and applying the Cauchy--Schwarz inequality along with the parameter-continuity assumption (A3), the combined estimate \eqref{lipschitz combined}, and the uniform coercivity bound \eqref{kappauniform}, we obtain:
\begin{align}\label{eqn:3.12}
\begin{split}
    \langle \op \xi_k,\xi_k\rangle_H
    +
    \underline\kappa_{\mathcal P}\norm{\xi_k}_V^2
    \leq\;&
    \langle \op \xi_k,\xi_k\rangle_H
    +
    a(\xi_k,\xi_k;\mu_1)
    \\ =\;&
    -
    \left(
        a(U_k^0,\xi_k;\mu_0)
        -
        a(U_k^0,\xi_k;\mu_1)
        +
        \langle F(U_k^0;\mu_0)-F(U_k^1;\mu_1),\xi_k\rangle_H
    \right)
    \\
    \leq\;&
    \abs{
        a(U_k^0,\xi_k;\mu_0)
        -
        a(U_k^0,\xi_k;\mu_1)
    }
    +
    \norm{
        F(U_k^0;\mu_0)-F(U_k^1;\mu_1)
    }_H
    \norm{\xi_k}_H
    \\
    \leq\;&
    \omega_a^{\mathcal P}(\delta)
    \norm{U_k^0}_V
    \norm{\xi_k}_V
    +
    \left(
        \omega_F^{\mathcal P}(\delta)(\norm{U_k^0}_V+1)
        +
        \overline L_{F,\mathcal P}\norm{\xi_k}_V
    \right)
    \norm{\xi_k}_H .
\end{split}
\end{align}

Employing the embedding \(\norm{\xi_k}_H^2\le \alpha\norm{\xi_k}_V^2\) and combining with inequality \eqref{eqn:3.12}, we proceed to apply Young's inequality:
\begin{align}\label{eqn:3.15}
\begin{split}
    \langle \op \xi_k,\xi_k\rangle_H
    +
    \underline\kappa_{\mathcal P}\norm{\xi_k}_V^2
    \leq\;&
    \omega_a^{\mathcal P}(\delta)
    \norm{U_k^0}_V
    \norm{\xi_k}_V
    +
    \left(
        \omega_F^{\mathcal P}(\delta)(\norm{U_k^0}_V+1)
        +
        \overline L_{F,\mathcal P}\norm{\xi_k}_V
    \right)
    \norm{\xi_k}_H
    \\
    \leq\;&
    \frac{\omega_a^{\mathcal P}(\delta)^2}
         {\underline\kappa_{\mathcal P}}
    \norm{U_k^0}_V^2
    +
    \frac{\underline\kappa_{\mathcal P}}{4}
    \norm{\xi_k}_V^2
    +
    \frac{\alpha\,\omega_F^{\mathcal P}(\delta)^2}
         {\underline\kappa_{\mathcal P}}
    \left(\norm{U_k^0}_V+1\right)^2
    +
    \frac{\underline\kappa_{\mathcal P}}{4}
    \norm{\xi_k}_V^2
    \\
    &\quad
    +
    \frac{\underline\kappa_{\mathcal P}}{2}
    \norm{\xi_k}_V^2
    +
    \frac{\overline L_{F,\mathcal P}^{\,2}}
         {2\underline\kappa_{\mathcal P}}
    \norm{\xi_k}_H^2
    \\
    \leq\;&
    \frac{\Lambda_{\mathcal P}}{2}
    \rho_{\mathcal P}(\delta)^2
    \left(\norm{U_k^0}_V^2+1\right)
    +
    \underline\kappa_{\mathcal P}\norm{\xi_k}_V^2
    +
    \frac{E_{\mathcal P}}{2}
    \norm{\xi_k}_H^2,
\end{split}
\end{align}
where
$    E_{\mathcal P}
    =
    \frac{\overline L_{F,\mathcal P}^{\,2}}
         {\underline\kappa_{\mathcal P}},
    \Lambda_{\mathcal P}
    =
    \frac{2+4\alpha}{\underline\kappa_{\mathcal P}}.$

Canceling the coercive term on both sides of \eqref{eqn:3.15}, and applying the Cauchy--Schwarz inequality to the discrete time derivative term, we derive:
\begin{equation}
    \label{eqn:3.16}
    \frac{1}{\Delta t}
    \left(
        \norm{\xi_k}_H^2
        -
        \norm{\xi_k}_H\norm{\xi_{k-1}}_H
    \right)
    \leq
    \langle \op \xi_k,\xi_k\rangle_H
    \leq
    \frac{\Lambda_{\mathcal P}}{2}
    \rho_{\mathcal P}(\delta)^2
    \left(\norm{U_k^0}_V^2+1\right)
    +
    \frac{E_{\mathcal P}}{2}
    \norm{\xi_k}_H^2 .
\end{equation}
Applying Young's inequality to the cross-term
    $\norm{\xi_k}_H\norm{\xi_{k-1}}_H
    \leq
    \frac12\norm{\xi_k}_H^2
    +
    \frac12\norm{\xi_{k-1}}_H^2$, we obtain the recursive estimate:
\begin{equation*}
    \norm{\xi_k}_H^2
    \leq
    \Lambda_{\mathcal P}\rho_{\mathcal P}(\delta)^2
    \Delta t
    \left(\norm{U_k^0}_V^2+1\right)
    +
    E_{\mathcal P}\Delta t\norm{\xi_k}_H^2
    +
    \norm{\xi_{k-1}}_H^2 .
\end{equation*}

Recalling the numerical stability condition \eqref{stability_cond} \(1-E_{\mathcal P}\Delta t\ge \frac12\), we observe that:
\begin{equation*}
    \frac{1}{1-E_{\mathcal P}\Delta t}
    =
    1+\frac{E_{\mathcal P}\Delta t}{1-E_{\mathcal P}\Delta t}
    \leq
    1+2E_{\mathcal P}\Delta t.
\end{equation*}
Then we rearrange the recursive inequality to obtain:
\begin{equation*}
    \norm{\xi_k}_H^2
    \leq
    (1+2E_{\mathcal P}\Delta t)
    \left(
        \Lambda_{\mathcal P}\rho_{\mathcal P}(\delta)^2
        \Delta t
        \left(\norm{U_k^0}_V^2+1\right)
        +
        \norm{\xi_{k-1}}_H^2
    \right).
\end{equation*}
Unfolding this recursion yields the discrete accumulation:
\begin{equation}\label{eqn:3.24}
    \norm{\xi_k}_H^2
    \leq
    \Lambda_{\mathcal P}\rho_{\mathcal P}(\delta)^2
    \Delta t
    \sum_{q=1}^k
    (1+2E_{\mathcal P}\Delta t)^{k+1-q}
    \left(\norm{U_q^0}_V^2+1\right)
    +
    (1+2E_{\mathcal P}\Delta t)^k
    \norm{\xi_0}_H^2 .
\end{equation}
Utilizing the elementary inequality \(1+x\le e^x\) for all \(x\ge0\), we bound the growth factors
    $(1+2E_{\mathcal P}\Delta t)^s
    \leq
    e^{2E_{\mathcal P}s\Delta t}
    \text{for any }s>0$. This allows us to estimate the summation term:
\begin{align*}
    \Lambda_{\mathcal P}\rho_{\mathcal P}(\delta)^2
    \Delta t
    \sum_{q=1}^k
    (1+2E_{\mathcal P}\Delta t)^{k+1-q}
    \left(\norm{U_q^0}_V^2+1\right)
    &
    \leq
    \Lambda_{\mathcal P}\rho_{\mathcal P}(\delta)^2
    \Delta t
    \sum_{q=1}^k
    e^{2E_{\mathcal P}(k+1-q)\Delta t}
    \left(\norm{U_q^0}_V^2+1\right)
    \\
    &
    =
    \Lambda_{\mathcal P}\rho_{\mathcal P}(\delta)^2
    e^{2E_{\mathcal P}\Delta t}
    \Delta t
    \sum_{q=1}^k
    e^{2E_{\mathcal P}(k-q)\Delta t}
    \left(\norm{U_q^0}_V^2+1\right)
    \\
    &
    \leq
    \Lambda_{\mathcal P}\rho_{\mathcal P}(\delta)^2
    e^{2E_{\mathcal P}T}
    \Delta t
    \sum_{q=1}^m
    e^{2E_{\mathcal P}(1-q)\Delta t}
    \left(\norm{U_q^0}_V^2+1\right),
\end{align*}
where we have used the bounds \(k\Delta t\le T\).

Finally, observing that the initial conditions satisfy
$    \langle U_0^0-U_0^1,\psi\rangle_H=0
    \quad
    \forall \psi\in V^\ell$,
we conclude \(\norm{\xi_0}_H^2=0\), which simplifies \eqref{eqn:3.24} to the desired estimate:
\begin{equation*}
    \norm{U_k^0-U_k^1}_H^2
    \leq
    \Lambda_{\mathcal P}\rho_{\mathcal P}(\delta)^2
    e^{2E_{\mathcal P}T}
    \Delta t
    \sum_{q=1}^m
    e^{2E_{\mathcal P}(1-q)\Delta t}
    \left(\norm{U_q^0}_V^2+1\right).
\end{equation*}
This completes the proof of the stability estimate for the discrete POD approximations.
\end{proof}

Combining the results of Lemma \ref{lem:uk-u1k} and Lemma \ref{lem:Uk-U1k}, we establish the following sensitivity analysis for the POD approximation error with respect to parameter perturbations.

\begin{theorem}\label{thm:sensitive analysis}
There exists a positive constant \(\widehat C>0\), such that the following sensitivity estimate holds:
\begin{equation}\label{eqn: sensitivity d1<d+cr}
    \left(
        \frac{1}{m}\sum_{k=1}^m
        \norm{u_1(t_k)-U_k^1}_H^2
    \right)^{\frac12}
    \leq
    \left(
        \frac{1}{m}\sum_{k=1}^m
        \norm{u_0(t_k)-U_k^0}_H^2
    \right)^{\frac12}
    +
    \widehat C \rho_{\mathcal P}(\delta),
\end{equation}
where  $\widehat C$ is independent of $\ell$, $m$, $\delta$, and of
the parameter pair $\mu_0,\mu_1\in\mathcal P$. 
\end{theorem}

\begin{proof}
 For each discrete time index \(k\in\{1,2,\dots,m\}\), the triangle inequality and the auxiliary estimates \eqref{inequality:uk-u1k} and \eqref{inequality:U_k-U^1_k} give
\begin{align*}
    \norm{u_1(t_k)-U_k^1}_H 
    \leq\;&
    \norm{u_1(t_k)-u_0(t_k)}_H
    +
    \norm{u_0(t_k)-U_k^0}_H
    +
    \norm{U_k^0-U_k^1}_H
    \\
    \leq\;&
    \norm{u_0(t_k)-U_k^0}_H
    +
    \rho_{\mathcal P}(\delta)
    \left(
        \Lambda_{\mathcal P}
        e^{E_{\mathcal P}T}
        \int_0^T
            e^{-E_{\mathcal P}s}
            \left(1+\norm{u_0(s)}_V^2\right)
        \diff s
    \right)^{\frac12}
    \\
    &\quad
    +
    \rho_{\mathcal P}(\delta)
    \left(
        \Lambda_{\mathcal P}
        e^{2E_{\mathcal P}T}
        \Delta t
        \sum_{q=1}^m
            e^{2E_{\mathcal P}(1-q)\Delta t}
            \left(\norm{U_q^0}_V^2+1\right)
    \right)^{\frac12}
    \\
    =\;&
    \norm{u_0(t_k)-U_k^0}_H
    +
    C_1\rho_{\mathcal P}(\delta),
\end{align*}
where the stability constant \(C_1\) is explicitly defined by
\begin{equation}
\label{C1}
\begin{aligned}
C_1
:=
\left(
        \Lambda_{\mathcal P}
        e^{E_{\mathcal P}T}
        \int_0^T
            e^{-E_{\mathcal P}s}
            \left(1+\norm{u_0(s)}_V^2\right)
        \diff s
\right)^{\frac12} +
\left(
        \Lambda_{\mathcal P}
        e^{2E_{\mathcal P}T}
        \Delta t
        \sum_{q=1}^m
            e^{2E_{\mathcal P}(1-q)\Delta t}
            \left(\norm{U_q^0}_V^2+1\right)
\right)^{\frac12}.
\end{aligned}
\end{equation}

We now show that this constant $C_1$ can be bounded independently of \(\ell\) and \(m\).

Define
$  M_{\mathcal P}
    :=
    \|f\|_{L^\infty(0,T;H)}
    +
    \sup_{\mu\in\mathcal P}\|F(0;\mu)\|_H $. Indeed, \(f\in C([0,T];H)\),
and for any fixed \(\mu_*\in\mathcal P\), condition (F2) gives
$    \|F(0;\mu)\|_H
    \le
    \|F(0;\mu_*)\|_H
    +
    \omega_F^{\mathcal P}(\operatorname{diam}\mathcal P),
    $ for any $\mu\in\mathcal P$.
So the quantity \(M_{\mathcal P}\) is finite. Let $ C_f
    :=
    \frac{2\alpha M_{\mathcal P}^2}
         {\underline\kappa_{\mathcal P}}$ .
         
Testing the continuous problem with \(u_\mu(t)\) and using coercivity,
(F1), the embedding \(\norm{v}_H^2\le \alpha\norm{v}_V^2\), and Young's
inequality gives, for any \(t\in(0,T)\),
\begin{align*}
    \frac12\frac{\diff}{\diff t}\norm{u_\mu(t)}_H^2
    +
    \underline\kappa_{\mathcal P}\norm{u_\mu(t)}_V^2
    &\le
    \norm{f(t)}_H\norm{u_\mu(t)}_H
    +
    \norm{F(u_\mu(t);\mu)}_H\norm{u_\mu(t)}_H
    \\
    &\le
    M_{\mathcal P}\norm{u_\mu(t)}_H
    +
    \overline L_{F,\mathcal P}
    \norm{u_\mu(t)}_V\norm{u_\mu(t)}_H
    \\
    &\le
    \frac{\underline\kappa_{\mathcal P}}{4}
    \norm{u_\mu(t)}_V^2
    +
    \frac{\alpha M_{\mathcal P}^2}
         {\underline\kappa_{\mathcal P}}
    +
    \frac{\underline\kappa_{\mathcal P}}{2}
    \norm{u_\mu(t)}_V^2
    +
    \frac{\overline L_{F,\mathcal P}^{\,2}}
         {2\underline\kappa_{\mathcal P}}
    \norm{u_\mu(t)}_H^2 .
\end{align*}
Equivalently,
\begin{equation}\label{eq:uniform-exact-energy}
    \frac{\diff}{\diff t}\norm{u_\mu(t)}_H^2
    +
    \frac{\underline\kappa_{\mathcal P}}{2}
    \norm{u_\mu(t)}_V^2
    \le
    E_{\mathcal P}\norm{u_\mu(t)}_H^2
    +
    C_f,
\end{equation}
By Gronwall's inequality,
\[
    \norm{u_\mu(t)}_H^2
    \le
    e^{E_{\mathcal P}T}
    \left(
        \norm{\phi}_H^2+C_fT
    \right)
   ,
    \qquad 0\le t\le T,\quad \mu\in\mathcal P .
\]
Integrating \eqref{eq:uniform-exact-energy} over \((0,T)\), we further obtain
\begin{equation}
\label{eq:uniform-exact-V-bound}
 \begin{aligned}
    \frac{\underline\kappa_{\mathcal P}}{2}
    \int_0^T \norm{u_\mu(t)}_V^2\,\diff t
    &\le
    \norm{\phi}_H^2
    -
    \norm{u_\mu(T)}_H^2
    +
    E_{\mathcal P}
    \int_0^T \norm{u_\mu(t)}_H^2\,\diff t
    +
    C_fT
    \\
    &\le
    \norm{\phi}_H^2
    +
     E_{\mathcal P}T e^{E_{\mathcal P}T}   \left(
        \norm{\phi}_H^2+C_fT
    \right)
    +
    C_fT =: M_{\rm ex}.
\end{aligned}   
\end{equation}
Hence the first term in \eqref{C1} is bounded and independent of $\ell$ and $m$.  We next derive the corresponding uniform bound for the reduced solutions part of $\eqref{C1}$. Testing the backward Euler--Galerkin scheme with \(\psi=U_k^\mu\) yields
\begin{align*}
    \frac{1}{2\Delta t}
    \left(
        \norm{U_k^\mu}_H^2
        -
        \norm{U_{k-1}^\mu}_H^2
    \right)
    +
    \underline\kappa_{\mathcal P}\norm{U_k^\mu}_V^2
    &\le
    \norm{f(t_k)}_H\norm{U_k^\mu}_H
    +
    \norm{F(U_k^\mu;\mu)}_H\norm{U_k^\mu}_H
    \\
    &\le
    M_{\mathcal P}\norm{U_k^\mu}_H
    +
    \overline L_{F,\mathcal P}
    \norm{U_k^\mu}_V\norm{U_k^\mu}_H
    \\
    &\le
    \frac{\underline\kappa_{\mathcal P}}{4}
    \norm{U_k^\mu}_V^2
    +
    \frac{\alpha M_{\mathcal P}^2}
         {\underline\kappa_{\mathcal P}}
    +
    \frac{\underline\kappa_{\mathcal P}}{2}
    \norm{U_k^\mu}_V^2
    +
    \frac{\overline L_{F,\mathcal P}^{\,2}}
         {2\underline\kappa_{\mathcal P}}
    \norm{U_k^\mu}_H^2 .
\end{align*}
Therefore,
\begin{equation}\label{eq:uniform-rom-recursion}
    \left(1-E_{\mathcal P}\Delta t\right)
    \norm{U_k^\mu}_H^2
    +
    \frac{\underline\kappa_{\mathcal P}}{2}
    \Delta t\norm{U_k^\mu}_V^2
    \le
    \norm{U_{k-1}^\mu}_H^2
    +
    C_f\Delta t.
\end{equation}
Under the stability condition \eqref{stability_cond},
 $   \frac{1}{1-E_{\mathcal P}\Delta t}
    \le
    1+2E_{\mathcal P}\Delta t
    \le
    e^{2E_{\mathcal P}\Delta t}$.

Since \(U_0^\mu\) is the \(H\)-orthogonal projection of \(\phi\) onto
\(V^\ell\), we have $ \norm{U_0^\mu}_H
    \le
    \norm{\phi}_H$ .
It follows from \eqref{eq:uniform-rom-recursion} and the discrete Gronwall inequality that
\[
    \max_{0\le k\le m}\norm{U_k^\mu}_H^2
    \le
    e^{2E_{\mathcal P}T}
    \left(
        \norm{\phi}_H^2+C_fT
    \right).
\]
Summing \eqref{eq:uniform-rom-recursion} over \(k=1,\ldots,m\) gives
\begin{equation}
\label{eq:uniform-rom-V-bound}
    \begin{aligned}
    \frac{\underline\kappa_{\mathcal P}}{2}
    \Delta t
    \sum_{k=1}^m\norm{U_k^\mu}_V^2
    &\le
    \sum_{k=1}^m
    \left(
        \norm{U_{k-1}^\mu}_H^2
        -
        \norm{U_k^\mu}_H^2
    \right)
    +
    E_{\mathcal P}\Delta t
    \sum_{k=1}^m\norm{U_k^\mu}_H^2
    +
    C_fT
    \\
    &\le
    \norm{\phi}_H^2
    +
    E_{\mathcal P}T e^{2E_{\mathcal P}T}
    \left(
        \norm{\phi}_H^2+C_fT
    \right)
    +
    C_fT =:
    M_{\rm rom},
\end{aligned}
\end{equation}
which is bounded independent of $\ell,m$. Combining \eqref{eq:uniform-exact-V-bound},
\eqref{eq:uniform-rom-V-bound} and substituting into \eqref{C1}, we obtain that $C_1$ is bounded by
\begin{align*}
    C_1
    &\le
    \left(
        \Lambda_{\mathcal P}
        e^{E_{\mathcal P}T}
        \left(
            T+\frac{2M_{\rm ex}}{\underline\kappa_{\mathcal P}}
        \right)
    \right)^{\frac{1}{2}}
    +
    \left(
        \Lambda_{\mathcal P}
        e^{2E_{\mathcal P}T}
        \left(
            T+\frac{2M_{\rm rom}}{\underline\kappa_{\mathcal P}}
        \right)
    \right)^{\frac{1}{2}}
    =:
    \widehat C,
\end{align*}
where $\widehat C$ is independent of $\mu_0,\mu_1,\ell,m$.

Taking the root-mean-square norm over \(k=1,\ldots,m\) and applying Minkowski's inequality gives
\[
\left(\frac1m\sum_{k=1}^m \|u_1(t_k)-U_k^1\|_H^2\right)^{\frac{1}{2}}
\le
\left(\frac1m\sum_{k=1}^m \|u_0(t_k)-U_k^0\|_H^2\right)^{\frac{1}{2}}
+\widehat C\rho_{\mathcal P}(\delta).
\]
This completes the proof.

\end{proof}

To complete the total error estimate, we now analyze the standard POD approximation error $\frac1m\sum_{k=1}^m\norm{u_0(t_k)-U_k^0}_H^2$ for the reference system. The theoretical framework here extends the approach from \cite{gu2021error} to the more general case of nonlinear evolution equations of the form \eqref{PDE:nonlinear} at the fixed reference parameter $\mu_0$. Furthermore, while the error estimates in \cite{gu2021error} utilize the equivalence of norms in the finite-dimensional POD subspace, which makes the error constants dependent on the POD dimension $\ell$, our analysis avoids this step. By using the coercivity of the bilinear form to bound the higher-order terms induced by the nonlinear operator, we ensure that our global error constants ($\tilde{C}$) remain independent of $\ell$. This dimension-independence is essential for establishing a robust sensitivity error estimate.

For the fixed reference parameter \(\mu_0\), we define the Ritz projection
operator $P^\ell:=P_{\mu_0}^\ell:V\to V^\ell$ by assigning to each  \(v\in V\) the unique element \(P^\ell v\in V^\ell\) that satisfies
\begin{equation}\label{eqn:ritz_projection}
    a(P^\ell v,\psi;\mu_0)
    =
    a(v,\psi;\mu_0),
    \qquad
    \forall \psi\in V^\ell .
\end{equation}
The operator \(P^\ell\) is well defined and bounded by the continuity and coercivity of
\(a(\cdot,\cdot;\mu_0)\) on \(V\). Moreover, it satisfies the stability
estimate
\begin{equation}\label{eqn:ritz_stability}
    \norm{P^\ell v}_V
    \le
    \frac{\beta(\mu_0)}{\kappa(\mu_0)}
    \norm{v}_V,
    \qquad
    \forall v\in V .
\end{equation}
Using the same argument as in Lemmas 3 and 4 of \cite{kunisch2001galerkin}, together with the quasi-optimality of the Ritz projection, we obtain the following POD projection estimate.

\begin{lemma}\label{lem:POD basis error}
The projection errors for the reference snapshots and their discrete time
derivatives satisfy the following bounds with respect to the Ritz projection
\(P^\ell=P_{\mu_0}^\ell\):
\begin{equation}\label{eqn:pod_ritz_solution_error}
    \frac1m\sum_{k=0}^m
    \norm{u_0(t_k)-P^\ell u_0(t_k)}_V^2
    \le
    3\left(\frac{\beta(\mu_0)}{\kappa(\mu_0)}\right)^2
    \sum_{j=\ell+1}^n\lambda_j,
\end{equation}
and
\begin{equation}\label{eqn:pod_ritz_derivative_error}
    \frac1m\sum_{k=1}^m
    \norm{\op u_0(t_k)-P^\ell\op u_0(t_k)}_V^2
    \le
    3\left(\frac{\beta(\mu_0)}{\kappa(\mu_0)}\right)^2
    \sum_{j=\ell+1}^n\lambda_j.
\end{equation}
Here \(\lambda_j\) denotes the \(j\)-th positive eigenvalue of the snapshot
correlation matrix $K$ associated with the POD ensemble constructed at
\(\mu_0\).
\end{lemma}

\begin{theorem}\label{thm:POD error analysis}
Let \(\mu_0\in\mathcal P\) be the reference parameter used to construct the
POD space \(V^\ell\). There exists a constant \(\tilde C>0\), depending on
\(u_0\), \(\alpha\), \(\beta(\mu_0)\), \(\kappa(\mu_0)\), \(L_F(\mu_0)\), and
\(T\), but independent of the POD dimension \(\ell\) and the temporal
discretization parameter \(m\), such that
\begin{equation}\label{eqn:reference_pod_error}
    \frac1m\sum_{k=1}^m
    \norm{u_0(t_k)-U_k^0}_H^2
    \le
    \tilde C
    \left(
        \norm{\phi-P^\ell\phi}_H^2
        +
        \sum_{j=\ell+1}^n\lambda_j
        +
        (\Delta t)^2
    \right).
\end{equation}
\end{theorem}

\begin{proof}
For each time index \(k=0,1,\ldots,m\), we decompose the error into two
components $u_0(t_k)-U_k^0=\varrho_k+\vartheta_k$, where
$\varrho_k:=u_0(t_k)-P^\ell u_0(t_k), 
    \vartheta_k:=P^\ell u_0(t_k)-U_k^0$.
This decomposition yields the estimate
\begin{equation}\label{eqn:4.4}
    \frac1m\sum_{k=1}^m
    \norm{u_0(t_k)-U_k^0}_H^2
    \le
    \frac2m\sum_{k=1}^m\norm{\vartheta_k}_H^2
    +
    \frac2m\sum_{k=1}^m\norm{\varrho_k}_H^2 .
\end{equation}
By Lemma \ref{lem:POD basis error}, the projection error term satisfies
\begin{equation}\label{eqn:4.5}
    \frac2m\sum_{k=1}^m\norm{\varrho_k}_H^2
    \le
    \frac{2\alpha}{m}\sum_{k=1}^m\norm{\varrho_k}_V^2
    \le
    6\alpha \left(\frac{\beta(\mu_0)}{\kappa(\mu_0)}\right)^2
    \sum_{j=\ell+1}^n\lambda_j,
\end{equation}
where the first inequality follows from the continuous embedding
\(V\hookrightarrow H\), and the second inequality follows from
Lemma \ref{lem:POD basis error}.

Now consider the discrete time derivative $\op\vartheta_k
    =
    {(\vartheta_k-\vartheta_{k-1})}/{\Delta t}.$ Subtracting the POD scheme \eqref{eqn:POD_scheme_mu_i} with \(i=0\) from the
continuous problem \eqref{PDE:nonlinear} at the parameter \(\mu_0\), and using
the Ritz projection identity, we obtain for any test function
\(\psi\in V^\ell\) the error equation
\begin{equation}\label{eqn:theta_error_equation}
    \langle \op\vartheta_k,\psi\rangle_H
    +
    a(\vartheta_k,\psi;\mu_0)
    =
    -\langle v_k,\psi\rangle_H
    -
    \langle F(u_0(t_k);\mu_0)-F(U_k^0;\mu_0),\psi\rangle_H,
\end{equation}
where
$    v_k
    :=
    \partial_t u_0(t_k)-P^\ell\op u_0(t_k)
    =
    w_k+\eta_k,$
with
\begin{equation*}
    w_k:=\partial_t u_0(t_k)-\op u_0(t_k),
    \qquad
    \eta_k:=\op u_0(t_k)-P^\ell\op u_0(t_k).
\end{equation*}
Selecting the test function \(\psi=\vartheta_k\in V^\ell\) and applying the
coercivity condition yields
\begin{equation}\label{eqn:theta_energy_initial}
\begin{aligned}
    \frac1{\Delta t}
    \left(
        \norm{\vartheta_k}_H^2
        -
        \langle \vartheta_k,\vartheta_{k-1}\rangle_H
    \right)
    +
    \kappa(\mu_0)\norm{\vartheta_k}_V^2
    \le
    &
    \norm{F(u_0(t_k);\mu_0)-F(U_k^0;\mu_0)}_H
    \norm{\vartheta_k}_H
    +
    \norm{\vartheta_k}_H\norm{v_k}_H .
\end{aligned}
\end{equation}

Using the cross-space Lipschitz condition \eqref{lipschitz in function} and
Young's inequality, we estimate the nonlinear term:
\begin{equation}\label{eqn:nonlinear_reference_estimate}
\begin{aligned}
\norm{F(u_0(t_k);\mu_0)-F(U_k^0;\mu_0)}_H
&\norm{\vartheta_k}_H
 \le 
L_F(\mu_0)\norm{u_0(t_k)-U_k^0}_V\norm{\vartheta_k}_H
\\
&\le
L_F(\mu_0)
\left(
    \norm{\vartheta_k}_V+\norm{\varrho_k}_V
\right)
\norm{\vartheta_k}_H
\\
&
\le
\frac{L_F(\mu_0)}{2}\norm{\varrho_k}_V^2
+
\kappa(\mu_0)\norm{\vartheta_k}_V^2
+
\left(
    \frac{L_F(\mu_0)}{2}
    +
    \frac{L_F(\mu_0)^2}{4\kappa(\mu_0)}
\right)
\norm{\vartheta_k}_H^2 .
\end{aligned}
\end{equation}

Define $   C_2
    :=
    L_F(\mu_0)
    +
    \frac{L_F(\mu_0)^2}{2\kappa(\mu_0)} $. Combining \eqref{eqn:theta_energy_initial} and
\eqref{eqn:nonlinear_reference_estimate}, and simplifying, we obtain
\begin{equation}\label{eqn:theta_before_recursion}
    \norm{\vartheta_k}_H^2
    -
    \langle \vartheta_k,\vartheta_{k-1}\rangle_H
    \le
    \Delta t
    \left(
        \norm{\vartheta_k}_H\norm{v_k}_H
        +
        \frac{L_F(\mu_0)}{2}\norm{\varrho_k}_V^2
        +
        \frac{C_2}{2}\norm{\vartheta_k}_H^2
    \right).
\end{equation}

Applying Young's inequality to the terms
\(\langle \vartheta_k,\vartheta_{k-1}\rangle_H\) and
\(\norm{\vartheta_k}_H\norm{v_k}_H\), we get
\begin{equation}\label{eqn:theta_stability_raw}
    \left(1-(1+C_2)\Delta t\right)\norm{\vartheta_k}_H^2
    \le
    \norm{\vartheta_{k-1}}_H^2
    +
    \Delta t
    \left(
        \norm{v_k}_H^2
        +
        L_F(\mu_0)\norm{\varrho_k}_V^2
    \right).
\end{equation}
By the stability condition \eqref{stability_cond}, we have $1-(1+C_2)\Delta t>\frac12$.
Set $C_3:=2(1+C_2)$.
Then \eqref{eqn:theta_stability_raw} gives the recursive estimate
\begin{equation}\label{eqn:vartheta iteration}
    \norm{\vartheta_k}_H^2
    \le
    (1+C_3\Delta t)
    \left(
        \norm{\vartheta_{k-1}}_H^2
        +
        \Delta t
        \left(
            \norm{v_k}_H^2
            +
            L_F(\mu_0)\norm{\varrho_k}_V^2
        \right)
    \right).
\end{equation}
Iterating this inequality from \(k=1\) to \(m\) and using
\((1+C_3\Delta t)^k\le e^{C_3T}\), we obtain
\begin{equation}\label{eqn:vartheta sum interation}
    \norm{\vartheta_k}_H^2
    \le
    e^{C_3T}
    \left(
        \norm{\vartheta_0}_H^2
        +
        \Delta t
        \sum_{j=1}^k
        \left(
            \norm{v_j}_H^2
            +
            L_F(\mu_0)\norm{\varrho_j}_V^2
        \right)
    \right).
\end{equation}

The initial error satisfies
\begin{equation}\label{eqn:initial_theta_bound}
\begin{aligned}
    \norm{\vartheta_0}_H^2
    &=
    \norm{P^\ell u_0(t_0)-U_0^0}_H^2
    =
    \norm{P^\ell\phi-U_0^0}_H^2
    =
    \norm{U_0^0}_H^2
    +
    \norm{P^\ell\phi}_H^2
    -
    2\langle U_0^0,P^\ell\phi\rangle_H
    \\
    &\le
    \norm{\phi}_H^2
    +
    \norm{P^\ell\phi}_H^2
    -
    2\langle \phi,P^\ell\phi\rangle_H
    =
    \norm{\phi-P^\ell\phi}_H^2,
\end{aligned}
\end{equation}
where we have used the fact that \(U_0^0\) is the \(H\)-orthogonal projection
of \(\phi\) onto \(V^\ell\).

Summing inequality \eqref{eqn:vartheta sum interation} over
\(k=1,\ldots,m\), we obtain
\begin{equation}\label{eqn:vartheta_sum_before_bounds}
\begin{aligned}
    \sum_{k=1}^m\norm{\vartheta_k}_H^2
    \le
    e^{C_3T}
    \Bigg(
        m\norm{\phi-P^\ell\phi}_H^2
        +
        T\sum_{j=1}^m
        \left(
            \norm{v_j}_H^2
            +
            L_F(\mu_0)\norm{\varrho_j}_V^2
        \right)
    \Bigg).
\end{aligned}
\end{equation}
Since \(v_j=w_j+\eta_j\), we further have
\begin{equation}\label{eqn:vartheta_sum_split_v}
\begin{aligned}
    \sum_{k=1}^m\norm{\vartheta_k}_H^2
    \le
    e^{C_3T}
    \Bigg(
        m\norm{\phi-P^\ell\phi}_H^2
        +
        T\sum_{j=1}^m
        \left(
            2\norm{w_j}_H^2
            +
            2\norm{\eta_j}_H^2
            +
            L_F(\mu_0)\norm{\varrho_j}_V^2
        \right)
    \Bigg).
\end{aligned}
\end{equation}

The temporal discretization error admits the bound
\begin{equation}\label{eqn:time_discretization_error}
\begin{aligned}
    \sum_{j=1}^m\norm{w_j}_H^2
    &=
    \sum_{j=1}^m
    \norm{\partial_t u_0(t_j)-\op u_0(t_j)}_H^2 =
    \sum_{j=1}^m
    \left\|
        \frac1{\Delta t}
        \int_{t_{j-1}}^{t_j}
        (t-t_{j-1})\partial_{tt}u_0(t)\diff t
    \right\|_H^2
    \\
    &\le
    \sum_{j=1}^m
    \frac1{(\Delta t)^2}
    \int_{t_{j-1}}^{t_j}
    (t-t_{j-1})^2\diff t
    \int_{t_{j-1}}^{t_j}
    \norm{\partial_{tt}u_0(t)}_H^2\diff t
    \\
    &\le
    \frac{\Delta t}{3}
    \int_0^T \norm{\partial_{tt}u_0(t)}_H^2\diff t.
\end{aligned}
\end{equation}

Applying Lemma \ref{lem:POD basis error} to bound the projection errors
\(\norm{\eta_j}_H^2\) and \(\norm{\varrho_j}_V^2\), we conclude
\begin{equation}\label{eqn:4.8}
\begin{aligned}
    \sum_{k=1}^m\norm{\vartheta_k}_H^2
    \le
    e^{C_3T}
    \Bigg(
        &
        m\norm{\phi-P^\ell\phi}_H^2
        +
        \frac{2T\Delta t}{3}
        \int_0^T \norm{\partial_{tt}u_0(t)}_H^2\diff t
        \\&
        +
        3\left(\frac{\beta(\mu_0)}{\kappa(\mu_0)}\right)^2Tm
        \left(
            2\alpha+L_F(\mu_0)
        \right)
        \sum_{j=\ell+1}^n\lambda_j
    \Bigg).
\end{aligned}
\end{equation}

Combining \eqref{eqn:4.4}, \eqref{eqn:4.5}, and \eqref{eqn:4.8}, we obtain
the final error bound  \eqref{eqn:reference_pod_error}.

\end{proof}

Combining the parametric sensitivity estimate in
Theorem \ref{thm:sensitive analysis} with the POD error estimate at the
reference parameter in Theorem \ref{thm:POD error analysis}, we obtain the
main error estimate for the parameterized POD reduced-order model.

\label{section: main thm}
\begin{theorem}\label{thm:main theorem}
Let \(\mu_0,\mu_1\in\mathcal P\),
and let \(V^\ell\) be the POD space constructed from the snapshots at the
reference parameter \(\mu_0\). 
Then the POD approximation at the query parameter \(\mu_1\), computed in the
same reduced space \(V^\ell\), satisfies
\begin{equation}\label{finalerror}
\begin{aligned}
    \left(
        \frac1m\sum_{k=1}^m
        \norm{u_1(t_k)-U_k^1}_H^2
    \right)^{1/2}
    \le
    \tilde C^{1/2}
    \left(
        \norm{\phi-P^\ell\phi}_H^2
        +
        \sum_{j=\ell+1}^n\lambda_j
        +
        (\Delta t)^2
    \right)^{1/2}
    +
    \widehat C\rho_{\mathcal P}(\delta).
\end{aligned}
\end{equation}
where  \(\tilde C \) and  \(\widehat  C\) are
independent of the POD dimension \(\ell\), the temporal discretization
parameter \(m\), and the perturbation magnitude
\(\delta\).
\end{theorem}

\begin{proof}

Taking square roots in Theorem~\ref{thm:POD error analysis} and substituting the resulting reference-parameter bound into Theorem~\ref{thm:sensitive analysis} yields \eqref{finalerror}. The stated independence of \(\tilde C\) and \(\widehat C\) is inherited from those two
theorems.

\end{proof}

\section{Numerical Results}
\label{section: numerical part}
In this section, we illustrate our theoretical sensitivity estimates through
numerical experiments. We consider three representative classes of evolution
problems: a convection--diffusion equation with homogeneous Dirichlet boundary
conditions, the mean-free viscous G-equation, and the mean-free viscous
G-equation with strain. The G-equation models are posed with periodic
boundary conditions and are formulated on the zero-mean space as in \cite{gu2021error}. 

Throughout this section, \(\mu\) denotes the scalar parameter varied in a
given test. The POD basis is constructed from snapshots at a reference
parameter \(\mu_0\), and the same reduced space is then reused to compute
POD-ROM solutions at query parameters \(\mu_1\). We write $\delta:=|\mu_1-\mu_0|$.

For each \(\mu_1\), we report the error between the
full-order solution and the POD-ROM solution,
\[
    E_{\rm POD}(\mu_1)
    :=
    \left(
        \frac1m\sum_{k=1}^m
        \norm{u_{\mu_1}(t_k)-U_k^{\mu_1}}_H^2
    \right)^{1/2},
\]
where \(U_k^{\mu_1}\) is computed in the POD space generated at
\(\mu_0\).  The parameter modulus \(\rho_{\mathcal P}\) relevant to each
test is identified in the corresponding examples. The numerical results
are used to compare the observed dependence of \(E_{\rm POD}(\mu_1)\) on
\(\delta\) with the modulus-controlled behavior predicted by the sensitivity
analysis. The verification that the model equations satisfy the structural
assumptions of the abstract theory is provided in
\Cref{section: Verification}.

\medskip 
\noindent \textbf{Example 1.}  We first consider the two-dimensional convection--diffusion equation on the
unit square \(\Omega=[0,1]^2\) with homogeneous Dirichlet boundary conditions 
\cite{WU2022114498} as a baseline benchmark problem:
\begin{equation}
\label{convection-diffusion}
\begin{cases}
\dfrac{\partial u}{\partial t}
+
\vec{b}\cdot\nabla u
-
\mu\Delta u
=0,
& \text{in } \Omega\times(0,T],
\\
u(\mathbf x,0)=f_0(\mathbf x),
& \text{in } \Omega,
\\
u(\mathbf x,t)=0,
& \text{on } \partial\Omega\times(0,T].
\end{cases}
\end{equation}
Here the parameter \(\mu>0\) is the diffusion coefficient, and the velocity
field is given by
 $   \vec b(x,y)
    =
    \bigl(\cos(2\pi y),\cos(2\pi x)\bigr)^\top .$
The
parameter dependence is linear in \(\mu\), so the corresponding parameter
modulus is Lipschitz-type
$ \rho_{\mathcal P}(\delta)=O(\delta)$.

For the numerical discretization, we use a uniform \(200\times200\) grid. The
diffusion term is approximated by the standard five-point stencil, and the
transport term is discretized by an upwind finite-difference scheme. Time
integration is performed using the backward Euler scheme with \(m=1000\)
time steps on the interval \([0,T]\), where \(T=1.0\). The initial condition is
chosen as $f_0(x,y)=(x^2-x)(y^2-y)$, which satisfies the homogeneous boundary condition. The POD basis is constructed from the full-order snapshots at the reference parameter $\mu_0=0.05$. We retain \(\ell=15\) POD modes, which capture more than \(99.9\%\) of the
snapshot energy.

The parameter-sensitivity data for this benchmark are reported in
\Cref{tab:CD_order}. Since the diffusion coefficient depends linearly on
\(\mu\), the associated modulus is Lipschitz-type, namely
\(\rho_{\mathcal P}(\delta)=O(\delta)\). The listed errors increase nearly
linearly with the perturbation size, in agreement with the predicted slope-one
behavior up to the reference POD error contribution. In particular, even at
the largest tested query parameter \(\mu_1=0.06\), corresponding to a
\(20\%\) perturbation of \(\mu_0=0.05\), the error remains
\(4.202\times10^{-4}\), indicating stable basis reuse for this benchmark.

\begin{table}[htbp]
\centering
\begingroup
\setlength{\tabcolsep}{2.5pt}
\renewcommand{\arraystretch}{1.08}
\resizebox{\linewidth}{!}{%
\begin{tabular}{@{}l*{9}{c}@{}}
\hline
\(\mu_1\)
& 0.05010 & 0.05020 & 0.05050 & 0.05100 & 0.05200
& 0.05300 & 0.05500 & 0.05750 & 0.06000 \\
\hline
\(\delta=|\mu_1-\mu_0|\)
& \(1.0\times10^{-4}\)
& \(2.0\times10^{-4}\)
& \(5.0\times10^{-4}\)
& \(1.0\times10^{-3}\)
& \(2.0\times10^{-3}\)
& \(3.0\times10^{-3}\)
& \(5.0\times10^{-3}\)
& \(7.5\times10^{-3}\)
& \(1.0\times10^{-2}\) \\

\(E_{\rm POD}(\mu_1)\)
& \(5.070\times10^{-6}\)
& \(1.035\times10^{-5}\)
& \(2.616\times10^{-5}\)
& \(5.201\times10^{-5}\)
& \(1.018\times10^{-4}\)
& \(1.491\times10^{-4}\)
& \(2.367\times10^{-4}\)
& \(3.343\times10^{-4}\)
& \(4.202\times10^{-4}\) \\

Fitted order \(\widehat p_{\rm POD}\)
& \multicolumn{9}{c}{\(\widehat p_{\rm POD}=0.9635\), predicted order \(p=1\)} \\
\hline
\end{tabular}%
}
\endgroup
\caption{
Observed POD-ROM errors for the convection--diffusion test. The POD basis is
constructed at \(\mu_0=0.05\) and reused at the listed query parameters
\(\mu_1\). 
}
\label{tab:CD_order}
\end{table}

\noindent \textbf{Example 2.}
As the second test case, we consider the viscous G-equation on the
two-dimensional periodic domain \(\Omega=[0,1]^2\). The G-equation is a
standard level-set model for flame-front propagation, where the front is
represented by the zero level set $\{(\mathbf x,t):G(\mathbf x,t)=0\}$ and is governed by the following Hamilton-Jacobi type equation:
\begin{equation}\label{eqn:viscous_G}
\begin{cases}
G_t+\vec V_A\cdot\nabla G+S_l|\nabla G|
=
dS_l\Delta G,
& \text{in }\Omega\times(0,T],
\\
G(\mathbf x,0)=\vec P\cdot\mathbf x,
& \text{in }\Omega,
\end{cases}
\end{equation}
where \(S_l=1.0\) is the laminar flame speed, \(d>0\) is the
Markstein number, and \(\vec P=(1,0)\) is the propagation direction. The cellular flow is written as $\vec V_A(\mathbf x)=A\vec V_0(\mathbf x)$,
with $\vec V_0(x,y) = \bigl( -\sin(2\pi x)\cos(2\pi y), \allowbreak \cos(2\pi x)\sin(2\pi y) \bigr)^\top$. 
Here $A>0$ denotes the flow amplitude.

Following the mean-free decomposition in \cite{gu2021error}, we write
$    G(\mathbf x,t)=\vec P\cdot\mathbf x+u(\mathbf x,t)$ and $u=\hat u+\bar u,$ where \(\hat u\) is periodic and mean-free, and
    $\int_\Omega \hat u(\mathbf x,t)\,\diff \mathbf x=0, \
    \bar u(t)=\int_\Omega u(\mathbf x,t)\,\diff\mathbf x$.
The mean-free formulation satisfies
\begin{equation}
\label{meanfree G}
\begin{cases}
    \hat u_t
    +
    \vec V_A\cdot\nabla\hat u
    -
    dS_l\Delta\hat u
    +
    S_l|\vec P+\nabla\hat u|
    -
    \displaystyle\int_{\Omega}S_l|\vec P+\nabla\hat u|\,\diff\mathbf x
    +
    \vec V_A\cdot\vec P
    =
    0,
    \\
    \bar u_t
    =
    -
    \displaystyle\int_{\Omega}S_l|\vec P+\nabla\hat u|\,\diff\mathbf x,
    \\
    \hat u(\mathbf x,0)=0,
    \qquad
    \bar u(0)=0.
\end{cases}
\end{equation}
The POD basis is constructed only from snapshots of the mean-free component
\(\hat u\). The mean component \(\bar u\) is recovered after \(\hat u\) is
computed. This formulation places \(\hat u\) in a zero-mean Hilbert space,
where the coercivity of the diffusion part follows from the
Poincar\'e--Wirtinger inequality. The verification of the corresponding
abstract assumptions is given in \Cref{section: Verification}.

For the numerical discretization, we use a uniform \(81\times81\) grid. The
advection term and the Hamilton--Jacobi term are discretized using WENO5-type
reconstructions with upwinding and a Godunov numerical Hamiltonian,
respectively. The diffusion term is discretized by a standard five-point
central difference stencil. Time integration is performed using the backward
Euler method with \(m=1000\) uniform time steps on \([0,T]\), where
\(T=1.0\). In all G-equation tests below, the reported error
\(E_{\rm POD}(\mu_1)\) is computed for the mean-free component.

\smallskip
\noindent\emph{Markstein number perturbation.}
We first take the Markstein number as the parameter, namely $\mu=d$,
while keeping \(A=4.0\), \(S_l=1.0\), and \(\vec P=(1,0)\) fixed. The POD
basis is constructed at the reference value $\mu_0=0.1$.

We retain \(\ell=12\) POD modes, which capture more than \(99.9\%\) of the
snapshot energy. Since $\mu$ enters the diffusion term linearly and the
remaining nonlinear term is independent of \(\mu\), the associated parameter
modulus is also Lipschitz-type $\rho_{\mathcal P}(\delta)=O(\delta)$.

The numerical results for this test are shown in \Cref{G1,G2}. In
\Cref{G1}, the  error \(E_{\rm POD}(\mu_1)\) is plotted
against \(\delta=|\mu_1-\mu_0|\). The plot
shows that the observed POD-ROM error is consistent with our theory. \Cref{G2} compares the
full-order and POD-ROM solutions at \(T=1.0\) for the reference parameter 
\(\mu_0=0.100\) and a query
 parameter \(\mu_1=0.115\). Notably, even when the perturbation magnitude reaches $15\%$ of the nominal value, the relative error is maintained on the order of $10^{-2}$, demonstrating the robustness of the reduced-order model.

\begin{figure}
    \centering
    \includegraphics[width=0.5\linewidth]{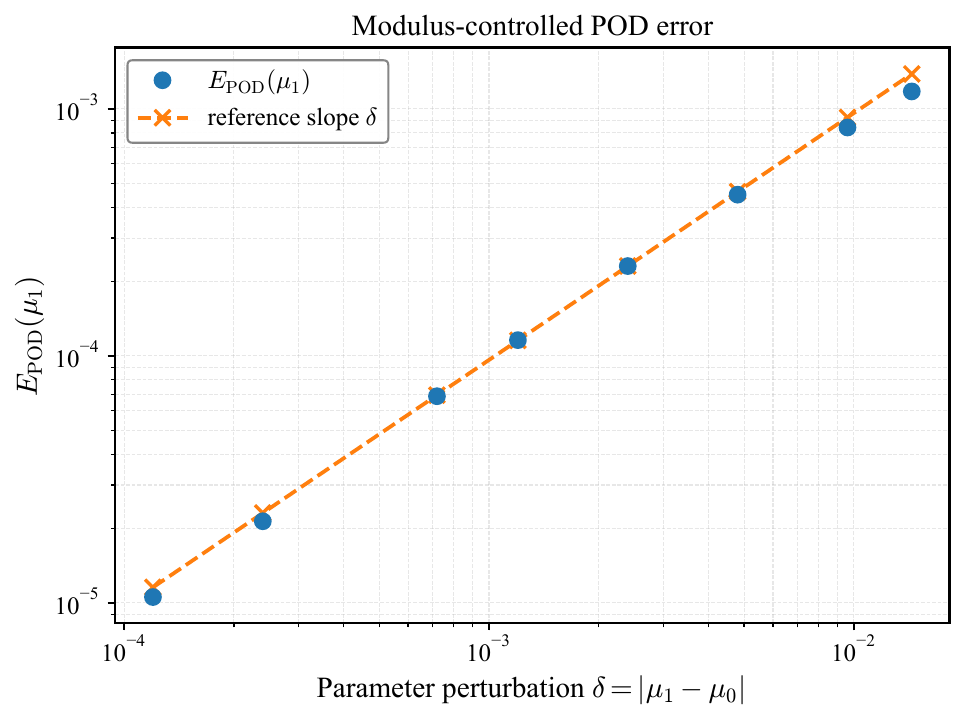}
    \caption{Modulus-controlled POD error for the mean-free viscous G-equation.
    Blue markers show \(E_{\rm POD}(\mu_1)\) versus
    \(\delta=|\mu_1-\mu_0|\). The dashed line is a
    reference slope proportional to \(\delta\), corresponding to the
    Lipschitz-type modulus \(\rho_{\mathcal P}(\delta)=O(\delta)\).}
    \label{G1}
\end{figure}

\begin{figure}
    \centering
    \includegraphics[width=0.8\linewidth]{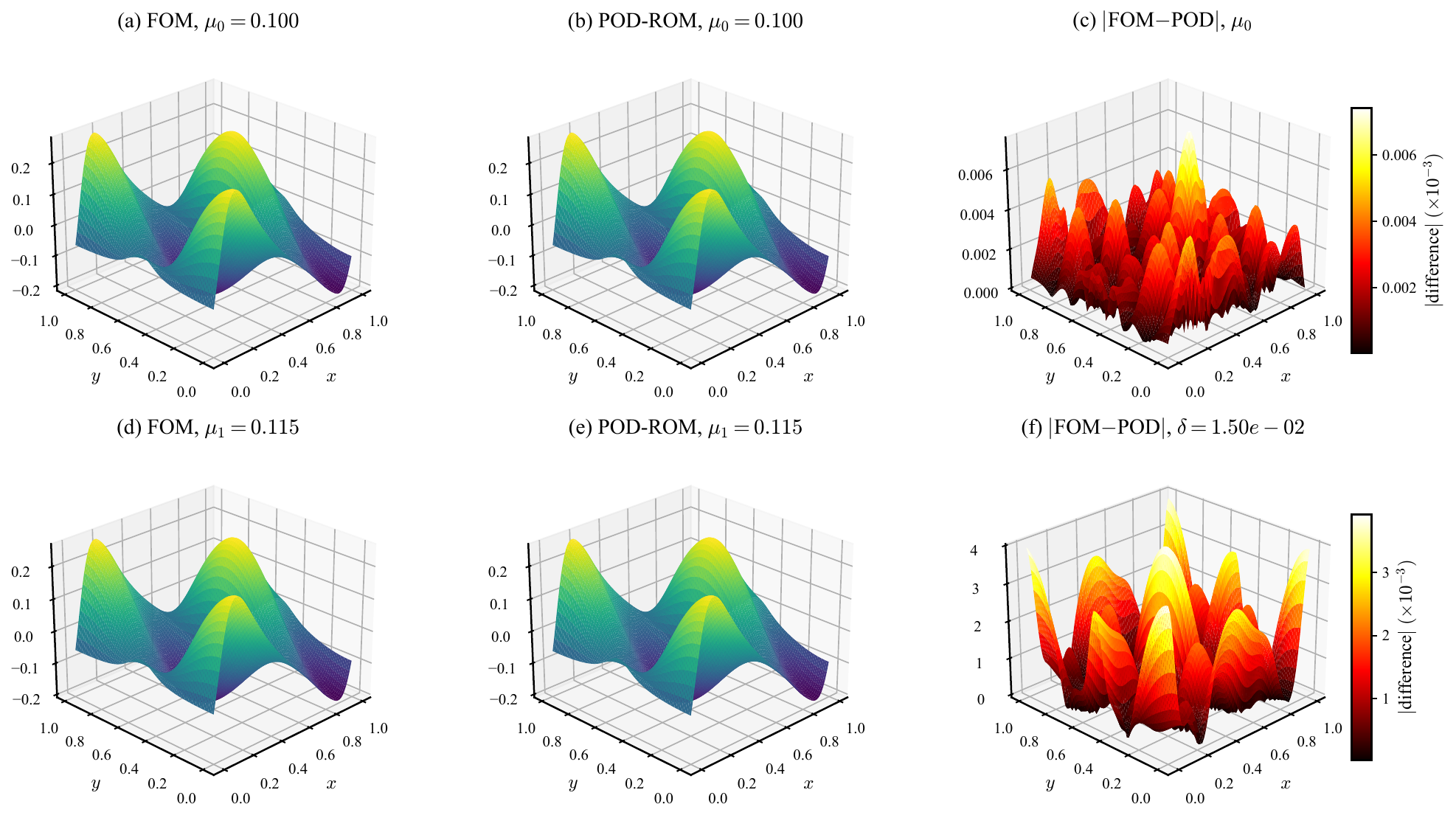}
    \caption{ Comparison of FOM and POD-ROM solutions for the mean-free formulation of the viscous G-equation at $T=1.0$. The top row shows the nominal case with $\mu_0 = 0.100$,  and the bottom row shows the perturbed case with a shift $\mu_1 = 0.115$ . The columns represent the FOM solution, the POD-ROM solution, and the absolute approximation error, respectively.}
    \label{G2}
\end{figure}

\smallskip
\noindent\emph{Nonlinear flow-amplitude parameterization.}
We next test a non-viscosity and nonlinear parameter in the same mean-free viscous
G-equation. In this experiment, the Markstein number is fixed at $d=0.1$ and the parameter \(\mu\) enters through the flow amplitude $A=A(\mu)$. Specifically, we set
$  A(\mu)
    =
    A_0+c_A|\mu-\mu_0|^{1/2}$, 
where  $A_0=4.0$, $\mu_0=1.0$ and $c_A=1.0$. This parameter
controls the strength of the cellular flow and enters both the transport term
\(\vec V_\mu\cdot\nabla\hat u\) and the lower-order term
\(\vec V_\mu\cdot\vec P\).

The results for this nonlinear amplitude test are reported in
\Cref{tab:G_nonlinA_order}.  Unlike the preceding Markstein-number perturbation, the
physical flow amplitude changes according to the nonlinear map
\(A(\mu)\).
Thus the dependence on the abstract parameter \(\mu\) is governed by the
H\"older modulus \(\rho_{\mathcal P}(\delta)=O(\delta^{1/2})\).
The table lists the nine query parameters, the corresponding amplitude
increments, and the POD-ROM errors \(E_{\rm POD}(\mu_1)\). The fitted rate $O(\delta^{0.4598})$  is close to the predicted half-order
scaling. The slight deviation from \(1/2\) is expected, since the reported
quantity is the total POD-ROM error and therefore includes the nonzero
reference-parameter POD error. These data are consistent with the
modulus-controlled estimate and illustrate that the parameter in the abstract theory need not be a viscosity, nor does it need to enter the model linearly.

For this multi-query test, the ten full-order trajectories, consisting of
the reference trajectory and the nine query trajectories listed in
\Cref{tab:G_nonlinA_order}, required
\(T_{\rm FOM}=695.73\) seconds, whereas the corresponding POD-ROM online
solves in the single reference space required \(T_{\rm ROM}=24.19\) seconds, giving an online speedup of about \(28.8\times\).

\begin{table}[t]
\centering
\begingroup
\setlength{\tabcolsep}{2.5pt}
\renewcommand{\arraystretch}{1.08}
\resizebox{\linewidth}{!}{%
\begin{tabular}{@{}l*{9}{c}@{}}
\hline
\(\mu_1\)
& 1.0001 & 1.0002 & 1.0005 & 1.0010 & 1.0020
& 1.0050 & 1.0100 & 1.0200 & 1.0500 \\
\hline
\(\delta=|\mu_1-\mu_0|\)
& \(1.0{\times}10^{-4}\)
& \(2.0{\times}10^{-4}\)
& \(5.0{\times}10^{-4}\)
& \(1.0{\times}10^{-3}\)
& \(2.0{\times}10^{-3}\)
& \(5.0{\times}10^{-3}\)
& \(1.0{\times}10^{-2}\)
& \(2.0{\times}10^{-2}\)
& \(5.0{\times}10^{-2}\) \\

\(A(\mu_1)-A_0\)
& \(1.000{\times}10^{-2}\)
& \(1.414{\times}10^{-2}\)
& \(2.236{\times}10^{-2}\)
& \(3.162{\times}10^{-2}\)
& \(4.472{\times}10^{-2}\)
& \(7.071{\times}10^{-2}\)
& \(1.000{\times}10^{-1}\)
& \(1.414{\times}10^{-1}\)
& \(2.236{\times}10^{-1}\) \\

\(E_{\rm POD}(\mu_1)\)
& \(2.006{\times}10^{-5}\)
& \(2.611{\times}10^{-5}\)
& \(3.838{\times}10^{-5}\)
& \(5.236{\times}10^{-5}\)
& \(7.221{\times}10^{-5}\)
& \(1.116{\times}10^{-4}\)
& \(1.560{\times}10^{-4}\)
& \(2.186{\times}10^{-4}\)
& \(3.418{\times}10^{-4}\) \\

Fitted order \(\widehat p_{\rm POD}\)
& \multicolumn{9}{c}{\(\widehat p_{\rm POD}=0.4598\), predicted order \(p=1/2\)} \\
\hline
\end{tabular}%
}
\endgroup
\caption{
Observed POD-ROM errors for the nonlinear flow-amplitude test in the
mean-free viscous G-equation. 
}
\label{tab:G_nonlinA_order}
\end{table}

\medskip 
\noindent \textbf{Example 3.}  Finally, we apply our framework to the viscous strain G-equation. This strain G-equation is particularly analytically challenging as it involves a highly non-convex Hamilton-Jacobi strain term \cite{LIU201320}. While the standard strain G-equation typically involves curvature-dependent motion, we consider a version where the curvature term is linearized into a Laplacian form $\mu S_l \Delta G$ to provide regular smoothing and numerical stability. We refer to this formulation as the viscous strain G-equation, defined on $\Omega=[0,1]^2$ with periodic boundary conditions:
\begin{equation}\label{eq:strain-G}
\begin{cases}
    G_t + \vec{V} \cdot \nabla G + (S_l + \mu\frac{\nabla G \cdot \nabla \vec{V} \cdot \nabla G}{|\nabla G|^2})|\nabla G| = \mu S_l\Delta G, & \text{in } \Omega \times (0,T], \\
G(\mathbf{x},0) = \vec{P} \cdot \mathbf{x} + f_0(x), & \text{in }   \Omega,
\end{cases}
\end{equation}
where $S_l=1.0$ is the laminar flame speed and $\mu$ represents the Markstein number serving as our interested parameter, $f_0(\mathbf{x}) = 0.1 \sin(2\pi x) \cos(2\pi y) + 0.05 \cos(4\pi x) \sin(4\pi y)$. The velocity field is a cellular flow given by $\vec{V} = A(-\cos(2\pi x)\sin(2\pi y), \sin(2\pi x)\cos(2\pi y))^\top$ with amplitude $A = 1.0$ and propagation direction $\vec{P} = (1,0)$.

Following the same mean-free decomposition used in Example 2, we reformulate the problem to align with our theoretical assumptions. This leads to the evolution equation for the zero-mean component \(\hat u\):
\begin{equation*}
\hat{u}_t + \vec{V} \cdot \nabla \hat{u}  + F_{strain}(\hat{u}; \mu) +\vec{V} \cdot \vec{P} = \mu S_l\Delta \hat{u},
\end{equation*}
where the nonlinear strain operator is defined as:
\begin{equation*}
F_{strain}(\hat{u}; \mu) = \sigma(\hat{u}; \mu)|\vec{P} + \nabla \hat{u}| - \int_{\Omega} \sigma(\hat{u}; \mu)|\vec{P} + \nabla \hat{u}| \diff\mathbf{x},
\end{equation*}
with the strain-dependent flame speed:
\begin{equation*}
\sigma(\hat{u}; \mu) = S_l + \mu\frac{(\vec{P}+\nabla \hat{u}) \cdot \nabla \vec{V} \cdot (\vec{P}+\nabla \hat{u})}{|\vec{P}+\nabla \hat{u}|^2}.
\end{equation*}
Note that now the initial condition $\hat u (\mathbf{x},0) = 0.1 \sin(2\pi x) \cos(2\pi y) + 0.05 \cos(4\pi x) \sin(4\pi y)$ is naturally mean-free. The numerical discretization remains consistent with the viscous G-equation. The POD basis is constructed at the nominal parameter $\mu_0= 0.05$ with $\ell = 15$ modes, capturing over 99.9\% of the system energy.

The numerical results are summarized in Figures \ref{fig:strainG2} and \ref{fig:strainG1}. As shown in Figure \ref{fig:strainG2}, the POD-ROM error exhibits a slope consistent with the $\rho_{\mathcal P}(\delta)=\mathcal{O}(\delta)$ convergence rate predicted by our main theorem, even in the presence of non-convex strain nonlinearities. Figure \ref{fig:strainG1} demonstrates that the POD basis constructed at the nominal parameter remains effective for the perturbed system, even as the strain term significantly modifies the flame front propagation. 

\begin{figure}[h]
    \centering
    \includegraphics[width=0.5\linewidth]{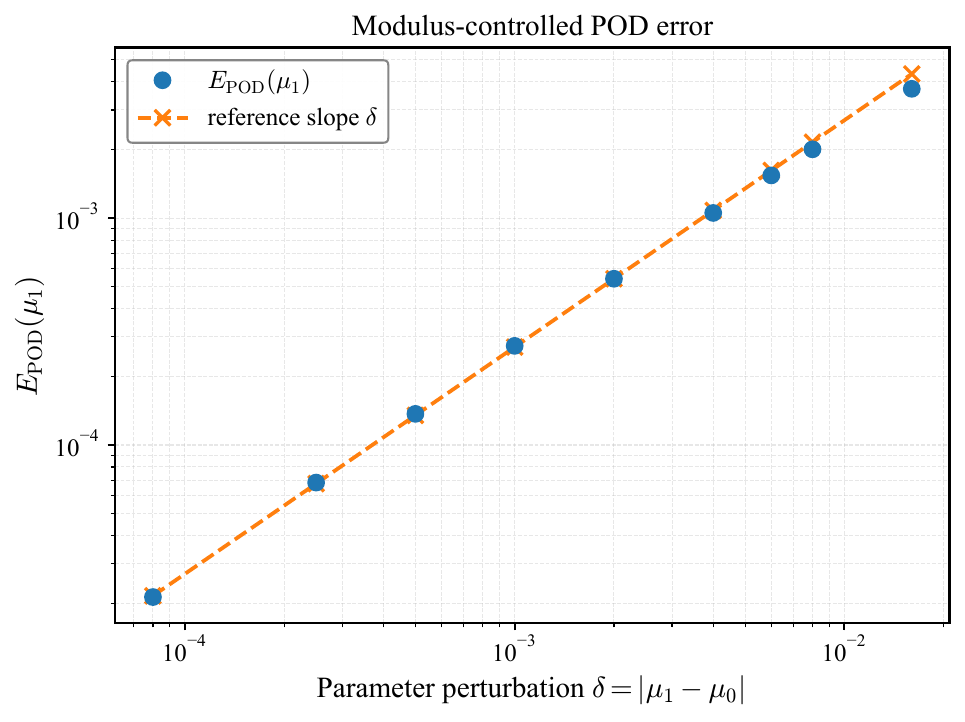}
    \caption{ Modulus-controlled POD error for the mean-free viscous strain G-equation.
    Blue markers show \(E_{\rm POD}(\mu_1)\) versus
    \(\delta=|\mu_1-\mu_0|\). The dashed line is a
    reference slope proportional to \(\delta\), corresponding to the
    Lipschitz-type modulus \(\rho_{\mathcal P}(\delta)=O(\delta)\).}
    \label{fig:strainG2}
\end{figure}

\begin{figure}[h]
    \centering
    \includegraphics[width=0.8\linewidth]{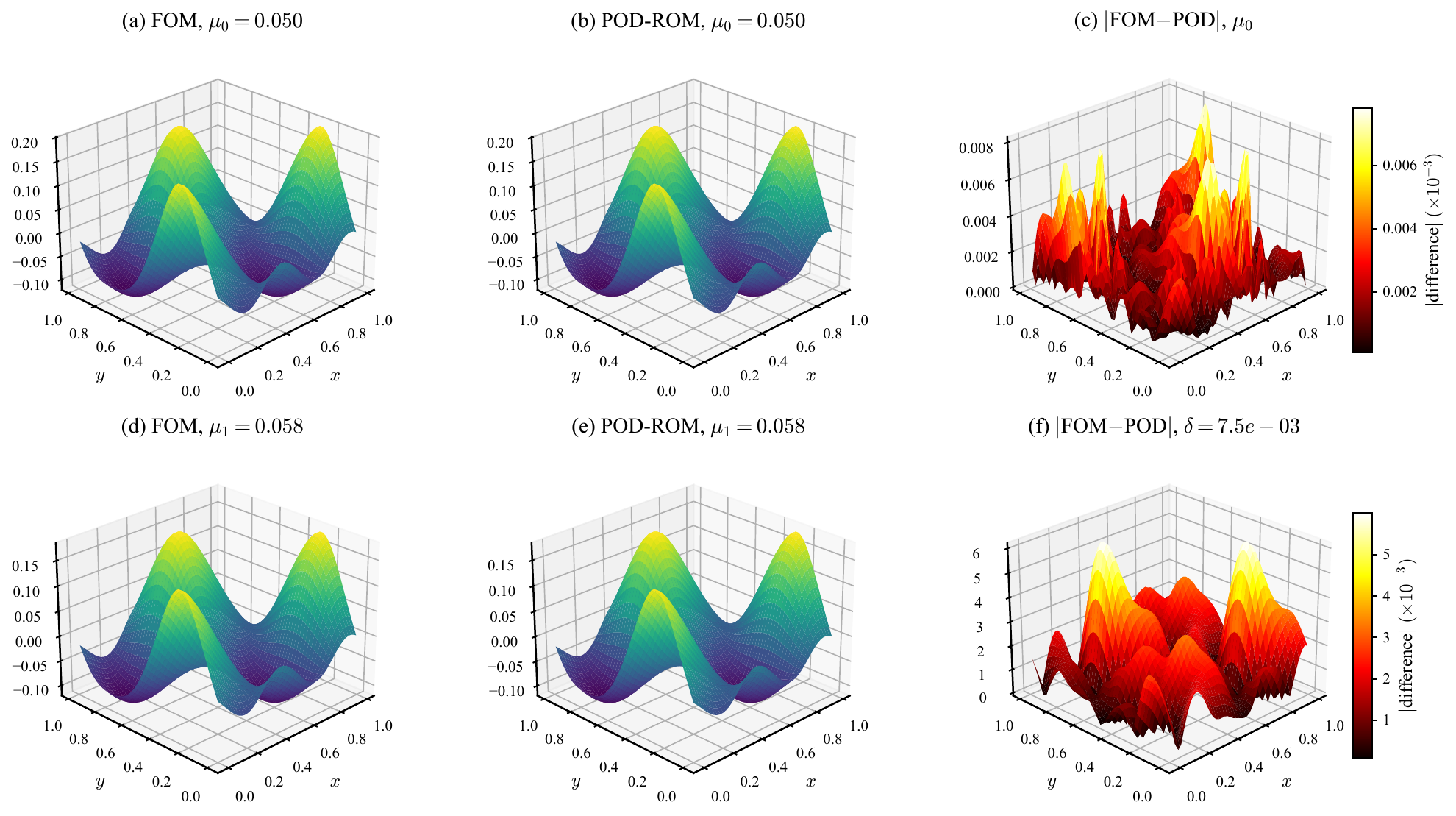}
    \caption{Comparison of the full-order and reduced-order solutions for the mean-free viscous strain G-equation at $T=1.0$. The top row shows the nominal case with $\mu_0 = 0.050$, and the bottom row shows the perturbed case with $\mu_1 = 0.058$. The columns represent the FOM solution, the POD-ROM solution, and the absolute approximation error, respectively.}
    \label{fig:strainG1}
\end{figure}

\section{Conclusions}
\label{sec:conclusions}
In this work, we developed a sensitivity analysis for POD reduced-order models applied to a class of parameterized semilinear evolution equations. The analysis is based on an abstract formulation involving a parameter-dependent bilinear form and a nonlinear operator satisfying a cross-space Lipschitz condition in the state variable. This cross-space condition is particularly useful for the gradient-dependent nonlinearities appearing in viscous G-equations and viscous strain G-equations. The parameter dependence is quantified through continuity moduli, which allows the final estimates to reflect the specific way in which the parameter enters each model. Within this setting, we showed that a POD basis constructed at a reference parameter can be reused at nearby query parameters, with the additional error controlled by the corresponding parameter modulus. The reference-parameter POD error estimate is obtained with constants independent of the POD dimension and time discretizations.

The numerical experiments are consistent with these theoretical predictions.  We tested the method on a convection-diffusion equation, the mean-free viscous G-equation, and the mean-free viscous strain G-equation. In the diffusion and Markstein-number perturbation tests, the observed POD-ROM errors are consistent with Lipschitz-type parameter dependence. The additional flow-amplitude test uses a non-viscosity parameter with a nonlinear parameterization, and the observed scaling follows the corresponding modulus. This observed agreement supports the modulus-dependent behavior predicted by the theory (Theorem \ref{thm:main theorem}) and demonstrates that a pre-computed POD basis can be robustly deployed within a local parameter neighborhood without unpredictable accuracy loss.

Some interesting directions of future work include extending the analysis to broader equation classes, such as the curvature G-equation \cite{LIU201320,OSHER198812}. Furthermore, our rigorous error bounds offer a solid mathematical basis for developing adaptive POD algorithms \cite{WU2022114498} that leverage these sensitivity estimates for dynamic basis updates. Finally, incorporating stochastic coefficients into the current framework will be an essential step to handle parametric uncertainty in applications such as turbulent combustion and climate modeling.

\appendix
\section{Verification of Theoretical Assumptions}
\label{section: Verification}

In this section, we verify that the numerical examples considered in
\Cref{section: numerical part} satisfy the structural assumptions used in
\Cref{Problem Settings}. This includes the continuity, coercivity, and
parameter-continuity conditions for the bilinear forms, as well as the
cross-space Lipschitz and parameter-continuity estimates for the nonlinear
operators.

For Example 1, we take \(H=L^2(\Omega)\) and \(V=H_0^1(\Omega)\). For Examples 2 and 3, after passing to the mean-free formulation, we take \[ H=L_0^2(\Omega) := \left\{ v\in L^2(\Omega): \int_\Omega v\,dx=0 \right\}, \qquad V=\mathring H_{\rm per}^1(\Omega) := H_{\rm per}^1(\Omega)\cap L_0^2(\Omega). \] The inner product on \(H\) is the usual \(L^2\)-inner product restricted to \(L_0^2(\Omega)\).  In all cases, $V$ is equipped with the norm $\norm{u}_V = \left( \norm{u}_{L^2(\Omega)}^2 + \norm{\nabla u}_{L^2(\Omega)}^2 \right)^{1/2}$.

\subsection{Verification of Bilinear Form Conditions (A1)--(A3)}
For the examples considered in this paper, the linear transport and diffusion terms are included in the parameter-dependent bilinear form. More precisely,
we write
\begin{equation}
    a(\varphi,\psi;\mu)
    =
    \int_\Omega (B_\mu\cdot\nabla\varphi)\psi\,\diff\mathbf{x}
    +
    D_\mu
    \int_\Omega \nabla\varphi\cdot\nabla\psi\,\diff\mathbf{x},
\end{equation}
where \(B_\mu\) is the prescribed incompressible transport field and
\(D_\mu>0\) is the diffusion coefficient. For the convection--diffusion
equation in Example 1, \(B_\mu=\mathbf b\) and \(D_\mu=\mu\). For the
Markstein-number perturbation in the mean-free viscous G-equation,
\(B_\mu=A\vec V_0\) and \(D_\mu=\mu S_l\). For the nonlinear flow-amplitude
test in the same viscous G-equation, \(B_\mu=A(\mu)\vec V_0\) and
\(D_\mu=dS_l\), where \(d\) is fixed and
    $A(\mu)=A_0+c_A|\mu-\mu_0|^\gamma, 0<\gamma<1$.
For the viscous strain G-equation, \(B_\mu=\vec V\) and \(D_\mu=\mu S_l\).

\noindent \textbf{Continuity (A1) and parameter continuity (A3):}
By applying the Cauchy--Schwarz inequality, we obtain
\begin{equation}
\begin{aligned}
    \abs{a(\varphi,\psi;\mu)}
    &\le
    \|B_\mu\|_{L^\infty(\Omega)}
    \|\nabla\varphi\|_{L^2(\Omega)}
    \|\psi\|_{L^2(\Omega)}
    +
    D_\mu
    \|\nabla\varphi\|_{L^2(\Omega)}
    \|\nabla\psi\|_{L^2(\Omega)}
    \\
    &\le
    \left(
        \|B_\mu\|_{L^\infty(\Omega)}+D_\mu
    \right)
    \|\varphi\|_V\|\psi\|_V .
\end{aligned}
\end{equation}
Thus condition (A1) is satisfied with $\beta(\mu)
    =
    \|B_\mu\|_{L^\infty(\Omega)}+D_\mu$.

Similarly, for \(\mu,\nu\in\mathcal P\), 
$    \abs{a(\varphi,\psi;\mu)-a(\varphi,\psi;\nu)}
    \le
    \left(
        \|B_\mu-B_\nu\|_{L^\infty(\Omega)}
        +
        |D_\mu-D_\nu|
    \right)
    \|\varphi\|_V\|\psi\|_V .$ So (A3) holds with the modulus
\begin{equation}\label{eq:omega-a-general}
    \omega_a^{\mathcal P}(s)
    :=
    \sup_{\substack{\mu,\nu\in\mathcal P\\ |\mu-\nu|\le s}}
    \left(
        \|B_\mu-B_\nu\|_{L^\infty(\Omega)}
        +
        |D_\mu-D_\nu|
    \right),
    \qquad 0\le s\le \operatorname{diam}(\mathcal P).
\end{equation}
For \(s>\operatorname{diam}(\mathcal P)\), we extend
\(\omega_a^{\mathcal P}\) by setting
$    \omega_a^{\mathcal P}(s)
    =
    \omega_a^{\mathcal P}(\operatorname{diam}(\mathcal P))$.
Since the maps
$    \mu\mapsto B_\mu\in L^\infty(\Omega),
    \mu\mapsto D_\mu\in\mathbb R$
are continuous on the compact set \(\mathcal P\), they are uniformly
continuous. Hence \(\omega_a^{\mathcal P}\) is nondecreasing,
\(\omega_a^{\mathcal P}(0)=0\), and $\lim_{s\to0}\omega_a^{\mathcal P}(s)=0$.
Moreover, for every \(\mu,\nu\in\mathcal P\),
\[
    \|B_\mu-B_\nu\|_{L^\infty(\Omega)}
    +
    |D_\mu-D_\nu|
    \le
    \omega_a^{\mathcal P}(|\mu-\nu|),
\]
which gives (A3).

In particular, the convection--diffusion equation admits the Lipschitz-type
choice \(\omega_a^{\mathcal P}(s)=s\). The Markstein-number perturbation in
the viscous G-equation and the parameter perturbation in the viscous strain
G-equation both admit the Lipschitz-type choice
\(\omega_a^{\mathcal P}(s)=S_l s\). For the nonlinear flow-amplitude test,
the amplitude satisfies
$  |A(\mu)-A(\nu)|
    \le
    c_A|\mu-\nu|^\gamma$,
and hence we can take
$   \omega_a^{\mathcal P}(s)
    =
    c_A\|\vec V_0\|_{L^\infty(\Omega)}s^\gamma $.

\noindent \textbf{Coercivity (A2):}
Setting \(\varphi=\psi\), we obtain
    $a(\psi,\psi;\mu)
    =
    \int_\Omega (B_\mu\cdot\nabla\psi)\psi\,\diff\mathbf{x}
    +
    D_\mu\norm{\nabla\psi}_{L^2(\Omega)}^2$.

Since \(B_\mu\) is incompressible, the transport term is skew-symmetric in the
energy estimate. Indeed,
\begin{equation}
    \int_\Omega (B_\mu\cdot\nabla\psi)\psi\,\diff\mathbf{x}
    =
    \frac12
    \int_\Omega B_\mu\cdot\nabla(\psi^2)\,\diff\mathbf{x}
    =
    0,
\end{equation}
where the boundary contribution vanishes either by the homogeneous Dirichlet
condition in Example 1 or by periodicity in the G-equation examples. Hence
    $a(\psi,\psi;\mu)
    =
    D_\mu\norm{\nabla\psi}_{L^2(\Omega)}^2 $.

For Example 1, the standard Poincar\'e inequality gives
\[
    \norm{\psi}_{L^2(\Omega)}^2
    \le
    C_P\norm{\nabla\psi}_{L^2(\Omega)}^2,
    \qquad
    \psi\in H_0^1(\Omega).
\]
For Examples 2 and 3, since \(\psi\in\mathring H_{\rm per}^1(\Omega)\) is
mean-free, the Poincar\'e--Wirtinger inequality gives the analogous bound.
Therefore,
\begin{equation}
    \norm{\psi}_V^2
    =
    \norm{\psi}_{L^2(\Omega)}^2
    +
    \norm{\nabla\psi}_{L^2(\Omega)}^2
    \le
    (C_P+1)\norm{\nabla\psi}_{L^2(\Omega)}^2 .
\end{equation}
Substituting this relation into the bilinear form gives
$    a(\psi,\psi;\mu)
    \ge
    \frac{D_\mu}{C_P+1}
    \norm{\psi}_V^2$.
Thus condition (A2) is satisfied with
    $\kappa(\mu)
    =
    \frac{D_\mu}{C_P+1}>0$.

\subsection{Verification of Nonlinear Operator Conditions (F1)--(F2)}
In this subsection, we verify the cross-space Lipschitz condition (F1) and the parameter-continuity condition (F2) for the nonlinear operators appearing in the numerical examples.
\\

\noindent \textbf{Example 1: Convection-Diffusion Equation.}  For the linear convection-diffusion equation \eqref{convection-diffusion}, the nonlinear operator vanishes identically, meaning $F(u; \mu) \equiv 0$. Consequently, both  conditions (F1)--(F2) hold trivially with $L_F(\mu)=0,  \omega_F^{\mathcal P}\equiv0$.
\\

\noindent \textbf{Example 2: Viscous G-Equation.}
For the mean-free viscous G-equation \eqref{meanfree G}, after the transport
and diffusion terms are included in the bilinear form, the remaining operator
can be written as
\[
    F(u;\mu)
    =
    E(u)
    -
    \int_{\Omega} E(u)\diff\mathbf{x}
    +
    B_\mu\cdot\vec P,
    \qquad
    E(u)=S_l\abs{\vec P+\nabla u}.
\]
Here \(B_\mu=A\vec V_0\) for the Markstein-number perturbation, while
\(B_\mu=A(\mu)\vec V_0\) for the nonlinear flow-amplitude test. To verify condition (F1), let \(u,v\in V\). For fixed \(\mu\), the term
\(B_\mu\cdot\vec P\) cancels in \(F(u;\mu)-F(v;\mu)\). By the reverse triangle
inequality,
$   \abs{E(u)-E(v)}
    \le
    S_l\abs{\nabla(u-v)}$.
    
Taking the \(H\)-norm of the difference \(F(u;\mu)-F(v;\mu)\) and using the
triangle inequality gives
\begin{equation}
    \norm{F(u;\mu)-F(v;\mu)}_H
    \le
    \norm{E(u)-E(v)}_H
    +
    \norm{
        \int_{\Omega} \left(E(u)-E(v)\right)\diff\mathbf{x}
    }_H .
\end{equation}
Since \(\abs{\Omega}=1\), the Cauchy--Schwarz inequality gives
\begin{align}
    \norm{F(u;\mu)-F(v;\mu)}_H
    &\le
    2\norm{E(u)-E(v)}_H \nonumber\le
    2S_l\norm{\nabla(u-v)}_{L^2(\Omega)}
    \le
    2S_l\norm{u-v}_V .
\end{align}
Thus (F1) holds with \(L_F(\mu)\equiv2S_l\).

For the Markstein-number perturbation, \(B_\mu=A\vec V_0\) is independent of
\(\mu=d\), and hence \(F\) is independent of the parameter. Therefore (F2)
holds with $\omega_F^{\mathcal P}\equiv0$.

For the nonlinear flow-amplitude test, \(B_\mu=A(\mu)\vec V_0\), where
$A(\mu)=A_0+c_A|\mu-\mu_0|^{\gamma}, 0< \gamma < 1$. For fixed \(u\in V\), the only parameter-dependent part of \(F(u;\mu)\) is
\(B_\mu\cdot\vec P\). Hence
\begin{align}
    \norm{F(u;\mu)-F(u;\nu)}_H
    &=
    \norm{(B_\mu-B_\nu)\cdot\vec P}_H \nonumber\le
    |A(\mu)-A(\nu)|
    \norm{\vec V_0\cdot\vec P}_H \nonumber \le
    c_A\norm{\vec V_0\cdot\vec P}_H
    |\mu-\nu|^\gamma .
\end{align}
Thus (F2) holds with the admissible modulus $\omega_F^{\mathcal P}(s)
    =
    c_A\norm{\vec V_0\cdot\vec P}_H s^\gamma $.
\\

\noindent \textbf{Example 3: Viscous Strain G-Equation.}
For the mean-free viscous strain G-equation, let
\(q_u:=\vec P+\nabla u\). The nonlinear term entering
\(F_{\rm strain}\) is the product
\begin{equation}
        \sigma(u;\mu)|q_u|
    =
    S_l|q_u|
    +
    \mu
    \frac{q_u\cdot\nabla\vec V\cdot q_u}{|q_u|},
\end{equation}
where the quotient is understood as zero whenever \(q_u=0\). This is the
continuous extension of the product
\(\sigma(u;\mu)|q_u|\) at such points.

We first record the following pointwise estimate. It can be proved by simple expansion and calculation. For any
\(p,q\in\mathbb R^2\),
\begin{equation}\label{eq:strain-quotient-lipschitz}
    \left|
        \frac{p\cdot\nabla\vec V\cdot p}{|p|}
        -
        \frac{q\cdot\nabla\vec V\cdot q}{|q|}
    \right|
    \le
    3\norm{\nabla\vec V}_{L^\infty(\Omega)}|p-q|.
\end{equation}

\textbf{Verification of (F1).}
Fix \(\mu\in\mathcal P\) and let \(u,v\in V\). Set
\(q_u:=\vec P+\nabla u\) and \(q_v:=\vec P+\nabla v\). By the reverse
triangle inequality and \eqref{eq:strain-quotient-lipschitz},
\[
\begin{aligned}
    \left|
        \sigma(u;\mu)|q_u|
        -
        \sigma(v;\mu)|q_v|
    \right|
    &\le
    S_l\big||q_u|-|q_v|\big|
    +
    \mu
    \left|
        \frac{q_u\cdot\nabla\vec V\cdot q_u}{|q_u|}
        -
        \frac{q_v\cdot\nabla\vec V\cdot q_v}{|q_v|}
    \right|
\\&\le
    \left(
        S_l
        +
        3\mu\norm{\nabla\vec V}_{L^\infty(\Omega)}
    \right)
    |\nabla(u-v)|.
\end{aligned}
\]
Since \(|\Omega|=1\), the mean-subtraction term satisfies
\[
    \left\|
        \int_\Omega
        \left[
            \sigma(u;\mu)|q_u|
            -
            \sigma(v;\mu)|q_v|
        \right]
        \diff\mathbf x
    \right\|_{L^2(\Omega)}
    \le
    \left\|
        \sigma(u;\mu)|q_u|
        -
        \sigma(v;\mu)|q_v|
    \right\|_{L^2(\Omega)} .
\]
Therefore,
\[
\begin{aligned}
    \norm{F_{\rm strain}(u;\mu)-F_{\rm strain}(v;\mu)}_H
    &\le
    2
    \left\|
        \sigma(u;\mu)|q_u|
        -
        \sigma(v;\mu)|q_v|
    \right\|_{L^2(\Omega)}
    \\
    &\le
    2
    \left(
        S_l
        +
        3\mu\norm{\nabla\vec V}_{L^\infty(\Omega)}
    \right)
    \norm{\nabla(u-v)}_{L^2(\Omega)}
    \\
    &\le
    2
    \left(
        S_l
        +
        3\mu\norm{\nabla\vec V}_{L^\infty(\Omega)}
    \right)
    \norm{u-v}_V .
\end{aligned}
\]
Thus (F1) holds with 
$L_F(\mu)
    =
    2
    \left(
        S_l
        +
        3\mu\norm{\nabla\vec V}_{L^\infty(\Omega)}
    \right)$.

\textbf{Verification of (F2).}
For fixed \(u\in V\), the dependence of
\(\sigma(u;\mu)|q_u|\) on \(\mu\) is affine. Hence
\[
    \left|
        \sigma(u;\mu)|q_u|
        -
        \sigma(u;\nu)|q_u|
    \right|
    =
    |\mu-\nu|
    \left|
        \frac{q_u\cdot\nabla\vec V\cdot q_u}{|q_u|}
    \right|
    \le
    |\mu-\nu|
    \norm{\nabla\vec V}_{L^\infty(\Omega)}
    |q_u|.
\]
Since \(q_u=\vec P+\nabla u\), we have
\(|q_u|\le |\vec P|+|\nabla u|\). Using the same mean-subtraction estimate
as above, we obtain
\[
\begin{aligned}
    \norm{F_{\rm strain}(u;\mu)-F_{\rm strain}(u;\nu)}_H
    &\le
    2|\mu-\nu|
    \left\|
        \frac{q_u\cdot\nabla\vec V\cdot q_u}{|q_u|}
    \right\|_{L^2(\Omega)}
    \\
    &\le
    2\norm{\nabla\vec V}_{L^\infty(\Omega)}
    |\mu-\nu|
    \left(
        \norm{\nabla u}_{L^2(\Omega)}
        +
        |\vec P|
    \right)
    \\
    &\le
    C_{\rm strain}|\mu-\nu|
    \left(
        \norm{u}_V+1
    \right),
\end{aligned}
\]
where $C_{\rm strain}
    :=
    2\norm{\nabla\vec V}_{L^\infty(\Omega)}
    \max\{1,|\vec P|\}$.
Therefore (F2) holds with the Lipschitz-type modulus
$    \omega_F^{\mathcal P}(s)
    =
    C_{\rm strain}s $.

\section*{Acknowledgments}
We would like to thank the research computing facilities of the Information Technology Services, the University of Hong Kong.

\bibliographystyle{siamplain}
\bibliography{references}
\end{document}